\documentclass[reqno, 11pt]{amsart}
\usepackage{url}
\usepackage[hidelinks]{hyperref}
\usepackage{cleveref}
\usepackage{helvet}
\usepackage{ stmaryrd }
\usepackage{color,soul}
\usepackage{amsthm, thmtools}
\usepackage[style=alphabetic]{biblatex}
\usepackage{caption} 
\usepackage{float}

\newtheorem{mainthm}{Theorem}

\newtheorem{mainth}{Theorem}

\newtheorem{theorem}{Theorem}[section]
\newtheorem{lemma}[theorem]{Lemma}
\newtheorem{prop}[theorem]{Proposition}

\newtheorem{con}[theorem]{Conjecture}

\newtheorem{remark}[theorem]{Remark}

\theoremstyle{definition}
\newtheorem{deff}[theorem]{Definition}

\newtheoremstyle{break}
  {\topsep}{\topsep}%
  {\itshape}{}%
  {\bfseries}{}%
  {\newline}{}%
\theoremstyle{break}

\numberwithin{equation}{section}

\usepackage{bbm}
\usepackage{euscript}
\usepackage{pb-diagram}
\usepackage{amsmath}

\usepackage{amsxtra}
\usepackage{amssymb}
\usepackage{pifont}
\usepackage{amsbsy}

\usepackage{graphicx}
\usepackage{epstopdf}
\usepackage{colortbl}
\usepackage{multirow}
\usepackage{hhline}

\newskip\aline \newskip\halfaline
\newcommand\bE{{\mathbb E}}

\newcommand{\bN}{{{\mathbb N}}}

\newcommand\bR{{\mathbb R}}
\newcommand\bS{{\mathbb S}}

\newcommand\bZ{{\mathbb Z}}

\DeclareMathOperator{\diam}{diam}

\DeclareMathOperator{\vol}{vol}

\newcommand{\eps}{{\epsilon}}

\def\inpr{\mathbin{\hbox to 6pt{\vrule height0.4pt width5pt depth0pt \kern-.4pt \vrule height6pt width0.4pt depth0pt\hss}}}

\def\XXint#1#2#3{{\setbox0=\hbox{$#1{#2#3}{\int}$ }
\vcenter{\hbox{$#2#3$ }}\kern-.59\wd0}}

\title{Examples for Scalar Sphere Stability}

\author{Paul Sweeney Jr}
\address{Department of Mathematics, Stony Brook University, Stony Brook, NY 11794, USA}
\email{paul.sweeney@stonybrook.edu}

\begin{document}
\maketitle

\begin{abstract}
The rigidity theorems of Llarull and Marques-Neves, which show two different ways scalar curvature can characterize the sphere, have associated stability conjectures. Here we produce the first examples related to these stability conjectures. The first set of examples demonstrates the necessity of including a condition on the minimum area of all minimal surfaces to prevent bubbling along the sequence. The second set of examples constructs sequences that do not converge in the Gromov-Hausdorff sense but do converge in the volume preserving intrinsic flat sense.  In order to construct such sequences, we improve the Gromov-Lawson tunnel construction so that one can attach wells and tunnels to a manifold with scalar curvature bounded below and only decrease the scalar curvature by an arbitrarily small amount. Moreover, we are able to generalize both the sewing construction of Basilio, Dodziuk, and Sormani, and the construction due to Basilio, Kazaras, and Sormani of an intrinsic flat limit with no geodesics.
\end{abstract}

\section{Introduction}

Rigidity theorems are often used to characterize manifolds in Riemannian geometry. A typical rigidity theorem says that if a Riemannian manifold satisfies some conditions, usually including a bound on curvature, then it must be isometric to a specific model geometry. One can naturally formulate a stability theorem from a rigidity theorem. A stability theorem says if the hypotheses of a rigidity theorem are perturbed, then the manifolds that satisfy these hypotheses are quantitatively close to the manifold characterized by the rigidity theorem. In this paper, we are concerned with stability theorems which are associated with rigidity theorems that characterize the sphere using a curvature bound on the scalar curvature.

The rigidity theorem of  Llarull \cite{Ll} and the rigidity theorem due to Marques and Neves \cite{MV} are two results that show how scalar curvature can characterize the sphere. These two rigidity theorems naturally give rise to stability conjectures. Below, we construct the first examples related to these stability conjectures. We demonstrate why a condition preventing bubbling is required, and we investigate different modes of convergence. In order to construct these examples, we prove an enhancement of the Gromov-Lawson tunnel construction \cite{GL} (see also Schoen-Yau \cite{SY}) which retains control over the scalar curvature.

First, let us recall Llarull's theorem \cite{Ll} which says that if there is a degree non-zero, smooth, distance non-increasing map from a closed, smooth, connected, Riemannian, spin, $n$-manifold, $M^n$, to the standard unit round $n$-sphere and the scalar curvature of $M^n$ is greater than or equal to $n(n-1)$, then the map is a Riemannian isometry. Gromov in \cite{Gr} proposed studying the stability question related to Llarull's rigidity theorem by investigating sequences of Riemannian manifolds $M^n_j=(M^n_j,g_j)$ with $\inf R^j\to n(n-1)$ and $Rad_{\bS^n}(M_j)\to 1$. $Rad_{\bS^n}(M^n)$ is the maximal radius $r$ of the $n$-sphere, $\bS^n(r)$, such that $M^n$ admits a distance non-increasing map from $M^n$ to $\bS^n(r)$ of non-zero degree. Based on this, Sormani \cite{S4} proposed the following stability conjecture. Before stating the conjecture, we recall the following definition:
\[
\mathrm{MinA}(M,g)=\inf \{|\Sigma|_g : \Sigma \text{ is a closed minimal hypersurface in } M\}.
\]
Also, throughout this paper we will condense notation and set $M_j^n=(M_j^n,g_j)$ when we have a sequence of Riemannian manifolds.

\begin{con}\label{conjLL}
Suppose $M^n_j=(M^n_j,g_j)$, $n\geq 3$, are closed smooth connected spin Riemannian manifolds such that
\[
R^j\geq n(n-1)-\frac{1}{j}, \text{ }\mathrm{MinA}(M^n_j)\geq A, \text{ }\diam\left( M^n_j\right) \leq D, \text{ } \vol\left( M^n_j\right) \leq V 
\]
where $R^j$ is the scalar curvature of $M^n_j$.
Furthermore, suppose there are smooth maps to the standard unit round $n$-sphere
\[
f_j:M^n_j\to \bS^n
\]
which are 1-Lipschitz and $\deg{f_j}\neq 0$.
Then $M^n_j$ converges in the $\mathcal{VF}$-sense to the standard unit round $n$-sphere.
\end{con}

We construct a sequence of manifolds each of which is two spheres connected by a thin tunnel, which is related to \Cref{conjLL}. This sequence converges in the volume-preserving intrinsic flat ($\mathcal{VF}$) sense to a disjoint union of two $n$-spheres (see \Cref{minafig}). Moreover, the sequence shows without the lower bound on $\mathrm{MinA}$ then the conclusion of \Cref{conjLL} fails to hold.

\begin{mainthm}\label{minALL}
There exists a convergent sequence of Riemannian manifolds $M^n_j=(\bS^n,g_j)$, $n\geq 3$, with $M_j^n\xrightarrow{\mathcal{VF}} M_\infty$
such that
\[
R^j\geq n(n-1)-\frac{1}{j}, \text{ }\diam\left( M_j\right) \leq D, \text{ and } \vol\left( M_j\right) \leq V,
\]
for some constants $D, V>0$. Furthermore, there are smooth degree one, $1$-Lipschitz maps $f_j:M^n_j\to (\bS^n,g_{rd})$ which converge to a $1$-Lipschitz map $f_\infty:M_\infty\to (\bS^n,g_{rd})$, and $M_\infty$ is the disjoint union of two $n$-spheres. 
\end{mainthm}

\begin{figure}[H]
    \centering
    \includegraphics[width=12cm]{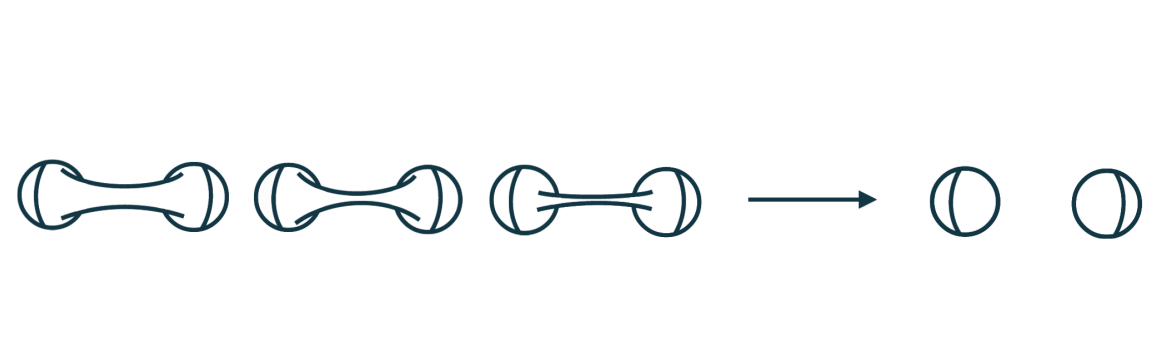}
    \caption{A sequence of spheres that converge in $\mathcal{VF}$-sense to the disjoint union of two spheres.}
    \label{minafig}
\end{figure}

Sormani proposed the $\mathrm{MinA}$ condition in \cite{S2} to prevent bad limiting behavior, such as bubbling and pinching, along the sequence. The motivation for such a condition comes from the sewing construction of Basilio, Dodziuk, and Sormani \cite{BDS}. This construction shows the existence of a sequence of manifolds with positive scalar curvature, which has an intrinsic flat $(\mathcal{F})$ limit that does not have positive scalar curvature in some generalized sense. Other sequences of positive scalar curvature manifolds have also been constructed (\cite{BS}, \cite{BKS}) whose $\mathcal{F}$-limits have undesirable properties. The key to the construction of these examples is the ability to glue in tunnels with controlled geometry. In those examples, it is unknown if the scalar curvature of the tunnel and of the resulting manifold can be kept close to the scalar curvature of the manifold to which the tunnel is being glued. Therefore, these examples may not satisfy the curvature condition in \Cref{conjLL}. In \Cref{constr}, we prove our two main technical propositions, which are of independent interest. One of which allows us to get quantitative control over the scalar curvature of the tunnel and of the resulting manifold. In particular, given a manifold with scalar curvature bounded below by $\kappa$, then for small enough $\eps>0$ there exists a tunnel such that the resulting manifold has scalar curvature bounded below by $\kappa-\eps$.

We use this new way of attaching tunnels to manifolds that maintains control over the scalar curvature to construct the sequence in \Cref{minALL}. Moreover, we can make a similar example related to Marques-Neves' rigidity theorem. The theorem of Marques-Neves pertains to the three dimensional sphere and the min-max quantity $\mathrm{width}$. Let us recall the definition of $\mathrm{width}$. Let $g$ be a Riemannian metric on the 3-sphere and $x_4:\bS^3\subseteq \bR^4 \to \bR$ be the height function. For each $t\in[-1,1]$, let ${\Sigma'}_t=\{x\in\bS^3:x_4=t\}$ and ${\Lambda'}$ be the collection of all families $\{\Sigma_t\}$ such that $\Sigma_t=F_t({\Sigma'}_t)$ for some smooth one-parameter family of diffeomorphisms $F_t$ of the 3-sphere all of which are isotopic to the identity. The $\mathrm{width}$ of $(\bS^3,g)$ is the following min-max quantity
\[
\mathrm{width}(\bS^3,g)=\inf_{\{\Sigma_t\}\in{\Lambda'}} \sup_{t\in[-1,1]} |\Sigma_t|_g,
\]
where $|\Sigma|_g$ is the Hausdorff two measure of $\Sigma$. 

The rigidity theorem of Marques-Neves \cite{MV} says if there is a Riemannian metric on the 3-sphere with positive Ricci curvature, scalar curvature greater than or equal to 6, and $\mathrm{width}(\bS^3,g) \geq 4\pi$, then it is isometric to the standard unit round 3-sphere. This leads to the following naive stability conjecture. 

\begin{con}\label{conjnMV}
Suppose $M^3_j=(\bS^3,g_j)$ are homeomorphic spheres satisfying
\[
R^j\geq 6-\frac{1}{j}, \text{ }\mathrm{width}(M^3_j)\geq 4\pi, \text{ }\diam\left( M^3_j\right) \leq D, \text{ and } \vol\left(M^3_j\right) \leq V 
\]
where $R^j$ is the scalar curvature of $M^3_j$. Then $M^3_j$ converges in the $\mathcal{VF}$-sense to $(\bS^3,g_{rd})$ where $g_{rd}$ is the Riemannian metric for the standard unit round 3-sphere.
\end{con}

In \cite{Mo}, Montezuma constructs Riemannian metrics $g_w$, $w>0$, on the $3$-spheres such that the scalar curvature is greater than or equal to 6 and the $\mathrm{width}(\bS^3,g_w)\geq w$. These manifolds look like a tree of spheres. In particular, they are constructed based on a finite full binary tree where the nodes are replaced with standard unit round $3$-spheres and the edges are replaced with Gromov-Lawson tunnels of positive scalar curvature. The width is shown to be proportional to the depth of the tree. Finally, by taking one of these manifolds with large enough width and scaling it scalar curvature greater than or equal to 6 is achieved. This example shows the failure of the rigidity statement of Marques-Neves and \Cref{conjnMV} when positive Ricci curvature is not assumed. 

Below we construct another counterexample that refutes \Cref{conjnMV} which is similar to the example constructed in \Cref{minALL}. By allowing the scalar curvature to be greater than or equal to $6-\eps$, we are to construct an example that is the connected sum of just two 3-spheres. Moreover, we are able to give explicit bounds on the volume and diameter of each manifold in the sequence. This counterexample is a sequence of spheres $M_j^3=(\bS^3,g_j)$ that converges in the volume preserving intrinsic flat $(\mathcal{VF})$ sense to the disjoint union of two spheres. The $M_j^3$ are two spheres connected by a thin tunnel (see \Cref{minafig}), and the tunnel gets increasingly thin along the sequence. 

\begin{mainth}[Counterexample to \Cref{conjnMV}]\label{minAMV}
There exists a convergent sequence of Riemannian manifolds $M^3_j=(\bS^3,g_j)$, with $M_j^3\xrightarrow{\mathcal{VF}} M_\infty$
such that
\[
R^j\geq 6-\frac{1}{j}, \text{ }\mathrm{width}(M^3_j)\geq 4\pi, \text{ }\diam\left( M^3_j\right) \leq D, \text{ and } \vol\left( M^3_j\right) \leq V,
\]
for some constants $D, V>0$, and $M_\infty$ is the disjoint union of two $3$-spheres.
\end{mainth}

Therefore, something stronger than $\mathrm{width}$ is required for a stability conjecture related to the rigidity theorem of Marques-Neves. A conjecture in \cite{S} attributed to Marques and Neves does hypothesize a stronger condition. In particular, it replaces the uniform lower bound on $\mathrm{width}$ with a uniform lower bound on $\mathrm{MinA}$.

\begin{con}\label{conjMV}
Suppose $M^3_j=(\bS^3,g_j)$ are homeomorphic spheres satisfying
\[
R^j\geq 6-\frac{1}{j}, \text{ }\mathrm{MinA}(M^3_j)\geq 4\pi-\frac{1}{j}, \text{ }\diam\left( M^3_j\right) \leq D, \text{ and } \vol\left( M^3_j\right) \leq V 
\]
where $R^j$ is the scalar curvature of $M^3_j$. Then $M^3_j$ converges in the $\mathcal{VF}$-sense to $(\bS^3,g_{rd})$ the standard unit round sphere.
\end{con}

Since $\mathrm{width}$ is achieved by a minimal surface, we have that $\mathrm{width}(M^3,g)\geq \mathrm{MinA}(M^3,g).$ Moreover, Marques-Neves show in \cite{MV} that if $(\bS^3,g)$ contains no stable minimal surfaces, then we have that $\mathrm{MinA}(\bS^3,g)= \mathrm{width}(\bS^3,g)$. In the proof of the Marques-Neves rigidity theorem, the hypothesis of positive Ricci curvature is used to ensure the manifold contains no stable minimal embedded spheres. By \cite[Appendix A]{MV}, we see that if both the scalar curvature of a $3$-manifold is sufficiently close to 6 and $\mathrm{MinA}$ is sufficiently close to $4\pi$ then the manifold contains no stable minimal embedded surfaces.

The sequence of Riemannian manifolds constructed in \Cref{minAMV} has $\mathrm{MinA}(M^n_j)\to 0$ and so does not satisfy the hypotheses of \Cref{conjMV}. Theorems \ref{minALL} and \ref{minAMV} show the necessity of including a hypothesis like the bound on $\mathrm{MinA}$ to prevent bubbling along the sequence.

When studying a stability conjecture related to scalar curvature, one also often considers examples similar to the example described by Ilmanen. Ilmanen first described this example to demonstrate that a sequence of manifolds of positive scalar curvature need not converge in the Gromov-Hausdorff $(\mathrm{GH})$ sense. The example is a sequence of spheres with increasingly many arbitrarily thin wells attached to them (see \Cref{spikesfig}). Sormani and Wenger \cite[Example A.7]{SW} showed, using their $\mathcal{F}$-convergence for integral currents, that the Ilmanen example converges in the $\mathcal{F}$-sense. Over the past decade, Ilmanen-like examples have been constructed in varying settings to demonstrate that $\mathrm{GH}$-convergence is not the appropriate convergence in which to ask stability conjectures related to scalar curvature (\cite{LaS}, \cite{L}, \cite{Pe}, \cite{LS2}, \cite{LeS}, \cite{LS1}, \cite{AP}, \cite{APS}). In these examples, it is unknown if one can attach a well and only decrease the scalar curvature by a small amount; consequently, it was unknown if Ilmanen-like examples could exist for \Cref{conjLL} and \Cref{conjMV}. 

Our other main technical proposition (see \Cref{constr} below) shows that one can attach a well to a manifold with scalar curvature bounded below and only decrease the scalar curvature by an arbitrarily small amount. Therefore, we are able to construct Ilmanen-like examples related to \Cref{conjLL} and \Cref{conjMV}. In particular, we construct a sequence of spheres with scalar curvature larger than $n(n-1)-\eps_j$, volumes and diameters bounded, and smooth maps to the unit round $n$-sphere which are 1-Lipschitz and $\deg{f_j}\neq 0$. This sequence does not converge in the $\mathrm{GH}$-sense but does converge in the volume above distance below $(VADB)$ sense and the $\mathcal{VF}$-sense. Likewise, we are able to construct a sequence of spheres $M^3_j$ with scalar curvature larger than $6-\eps_j$, $\mathrm{width}$ larger than $4\pi$, and volumes and diameters bounded that does not converge in the $\mathrm{GH}$-sense but does converge in the $\mathrm{VADB}$-sense and the $\mathcal{VF}$-sense. Therefore, we can construct Ilmanen-like examples related to \Cref{conjMV} and \Cref{conjLL}. We, however, cannot verify that $\mathrm{MinA}$ stays uniformly bounded from below even though we expect that it does.

\begin{mainthm}\label{spikesLL}
There exists a convergent sequence of Riemannian manifolds $M^n_j=(\bS^n,g_j)$, with $M_j^n\xrightarrow{\mathrm{VADB}}  M_\infty$ and $M_j^n\xrightarrow{\mathcal{VF}} M_\infty$
such that
\[
R^j\geq n(n-1)-\frac{1}{j}, \text{ }\diam\left( M^n_j\right) \leq D, \text{ and } \vol\left( M^n_j\right) \leq V,
\]
for some constants $D, V>0$, and $M_\infty$ is the $n$-sphere. Furthermore, there are smooth degree non-zero, $1$-Lipschitz maps $f_j:M^n_j\to (\bS^n,g_{rd})$, and $M_\infty$ is the standard unit round $n$-sphere. However, the sequence has no subsequence that converges in the $\mathrm{GH}$-sense. 

\end{mainthm}
\begin{mainth}\label{spikesMV}
There exists a convergent sequence of Riemannian manifolds $M^3_j=(\bS^3,g_j)$, with $M_j^3\xrightarrow{\mathrm{VADB}}  M_\infty$ and $M_j^3\xrightarrow{\mathcal{VF}} M_\infty$
such that
\[
R^j\geq 6-\frac{1}{j}, \text{ }\mathrm{width}(M^3_j)\geq 4\pi, \text{ }\diam\left( M^3_j\right) \leq D, \text{ and } \vol\left( M^3_j\right) \leq V,
\]
for some constants $D, V>0$, and $M_\infty$ is the standard unit round $3$-sphere. However, the sequence has no convergent subsequence in the $\mathrm{GH}$-topology.
\end{mainth}

In \cite[Remark 9.4]{S}, Sormani suggests that it is believable that someone can construct a sequence of spheres with increasingly many increasingly thin wells which satisfy the hypothesis of \Cref{conjMV}. \Cref{spikesMV} partially answers this question by constructing such a sequence that satisfies all the hypotheses of \Cref{conjMV} except the bound on $\mathrm{MinA}$. 

\begin{figure}[H]
    \centering
    \includegraphics[width=10cm]{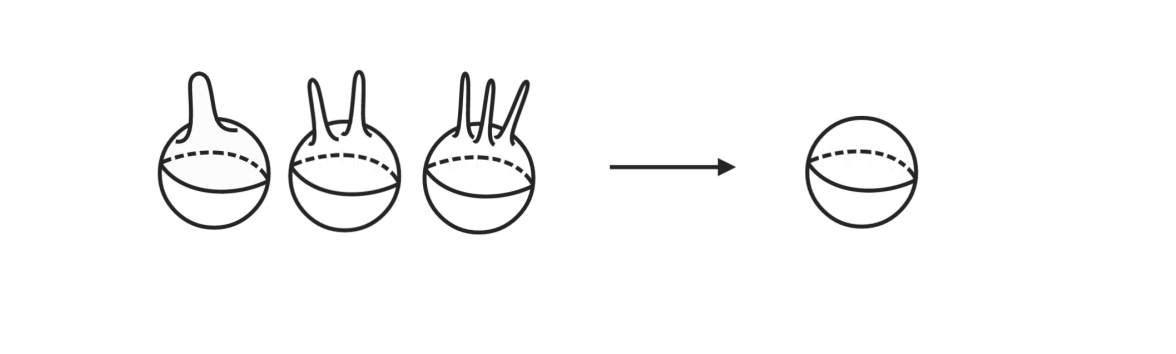}
    \caption{A sequence of spheres with increasingly many thin wells that converges in the $\mathrm{VADB}$-sense and $\mathcal{VF}$-sense to a sphere but has no convergent subsequence in the $\mathrm{GH}$-topology}
    \label{spikesfig}
\end{figure}

The main tools to prove the above theorems are new construction propositions which are proved in \Cref{constr}. We adapt the bending argument of Gromov and Lawson in \cite{GL}. Originally, the construction in \cite{GL} was used to make tunnels of positive scalar curvature to show, for example, that the connected sum of two manifolds with positive scalar curvature carries a metric of positive scalar curvature. For manifolds with constant positive sectional curvature, Dodziuk, Basilio, and Sormani in \cite{BDS} refined the construction to give control over the volume and diameter of the tunnel while maintaining positive scalar curvature.  Dodziuk in \cite{D} further refined the construction by replacing the positive sectional curvature condition with positive scalar curvature. In this paper, we construct wells and tunnels such that, if the scalar curvature of a manifold is bounded below, then one can attach a well or tunnel and only decrease the lower bound by an arbitrarily small amount while maintaining bounds on the diameter and volume.

The new well construction allows us to generalize the construction of Sormani and Wenger \cite[Example A.11]{SW} of a sequence of manifolds that converge in the  $\mathcal{F}$-sense to space that is not precompact. In particular, we are able to construct a sequence of spheres with scalar curvatures greater than $\kappa\geq 0$, uniformly bounded diameters, and uniformly bounded volumes such that the sequence converges in the $\mathcal{VF}$-sense to a limit that is not precompact. To construct the sequence we attach a sequence of increasingly thin wells to a sphere (see \Cref{noncptfig}).

\begin{mainthm}\label{noncompact}
There exists a convergent sequence of Riemannian manifolds $M^n_j=(\bS^n,g_j)$, $n\geq 3$, with $M_j^n\xrightarrow{\mathcal{VF}} M_\infty$
such that
\[
R^j\geq \kappa, \text{ }\diam\left( M^n_j\right) \leq D, \text{ and } \vol\left( M^n_j\right) \leq V,
\]
for some nonnegative constants $\kappa, D, V$, and $M_\infty$ is not precompact. 
\end{mainthm}
\begin{figure}[H]
    \centering
    \includegraphics[width=10cm]{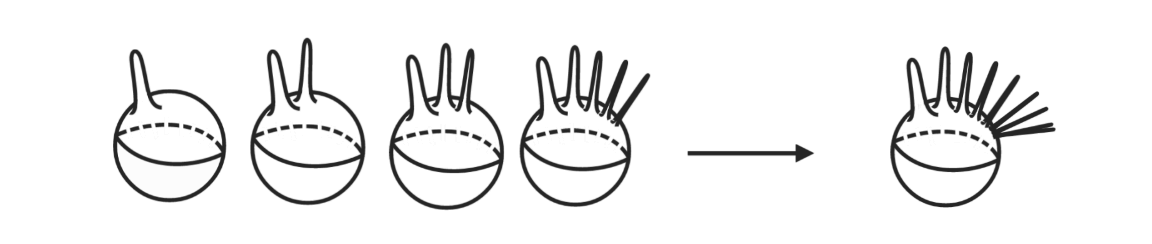}
    \caption{A sequence of spheres with increasingly many thin wells that converges in the $\mathcal{VF}$-sense to a limit which is not precompact.}
    \label{noncptfig}
\end{figure}
The new tunnel construction allows us to extend the sewing construction in \cite{BDS} and \cite{BS} to a more general setting. Basilio, Dodziuk, and Sormani \cite{BDS} used sewing manifolds to investigate the following question of Gromov which asks: What is the weakest notion of convergence such that a sequence of Riemannian manifolds, $M^n_j$ with scalar curvature $R^j\geq \kappa$ subconverges to a limit $M_\infty$ which may not be a manifold but has scalar curvature greater than $\kappa$ in some suitably generalized sense? They were able to show that when $\kappa=0$ there is a sequence of Riemannian manifolds with non-negative scalar curvature whose limit fails to have non-negative generalized scalar curvature where generalized scalar curvature is defined as
\begin{equation}\label{weakscal}
    wR(p_0):=\lim_{r\to 0} 6(n+2) \frac{\vol_{\bE^n}{B(0,r)}-\mathcal{H}^n(B(p_0,r))}{r^2\cdot\vol_{\bE^n}{B(0,r)} } \geq 0
\end{equation}
for the limit space.
\begin{remark}
For a Riemannian manifold $(M^n,g)$ with scalar curvature $R$, we see for all $p\in M^n$ that $wR(p)=R(p)$.
\end{remark}
We are able to provide a similar answer to Gromov's question for any $\kappa$. In particular, for any $\kappa$, there exists a sequence of increasingly tightly sewn manifolds all of which have scalar curvature greater than $\kappa$. Furthermore, this sequence of increasingly tightly sewn manifolds will converge in the $\mathcal{F}$-sense to a pulled metric space (see  \cite[Section 2]{BS} for discussion of such spaces) which fail to have generalized scalar curvature greater than or equal to $\kappa$ at the pulled point.
\begin{mainthm}\label{sewingthm}
 There exists a sequence of manifolds $M_j^n=(M^n,g_j)$ with scalar curvature $R^j\geq \kappa-\frac{1}{j}$ which converges in the $\mathcal{F}$-sense to a metric space $M_\infty$. Moreover, there is a point $p_0\in M_\infty$ such that
\begin{equation}
    wR(p_0):=\lim_{r\to 0} 6(n+2) \frac{\vol_{\bE^n}{B(0,r)}-\mathcal{H}^n(B(p_0,r))}{r^2\cdot\vol_{\bE^n}{B(0,r)} }=-\infty
\end{equation}
\end{mainthm} 
Lastly, the new tunnel construction allows us to generalize the construction of Basilio, Kazaras, and Sormani \cite{BKS}. They use long thin tunnels with positive scalar curvature to construct a sequence of manifolds that converges in the $\mathcal{F}$-sense to a space with no geodesics. Similarly, for any $\kappa>0$, we are able to construct a sequence of manifolds with scalar curvature bounded below by $\kappa$ whose limit is not a geodesic space.
\begin{mainthm}\label{nogeothm}
There is a sequence of closed, oriented, Riemannian manifolds $(M^n_j,g_j)$, $n\geq 3,$ with scalar curvature $R^j> \kappa>0$ such that the corresponding integral current spaces converge in the intrinsic flat sense to 
$$
M_\infty = \left(N, d_{\bE^{n+1}}, \int_{N}\right),
$$
where $N$ is the round $n$-sphere of curvature $\frac{2\kappa}{n(n-1)}$ and $d_{\bE^{n+1}}$ is the Euclidean distance induced from the standard embedding of $N$  into $\bE^{n+1}$. Furthermore, $M_\infty$ is not locally geodesic.
\end{mainthm}

\begin{table}[H]
\definecolor{lg}{gray}{0.9}
\newcolumntype{g}{>{\columncolor{lg}}c}
    \centerline{
        \begin{tabular}{|c|g|g|g|g|}
       
    \hline
        \rowcolor{white}
        &
        \textsc{Properties of}&
        {\textsc{Property of the}}&
        {\textsc{Type of}}&
        {\textsc{Does}}\\
        \rowcolor{white}
        &
        {\textsc{$M_j$}}&
        {\textsc{limit, $M_\infty$}}&
        {\textsc{convergence}}&
        {\textsc{$\mathrm{MinA}_j\to0$?}}\\
       
    \hline
         && \text{Shows necessity of}&&\\
         &&\text{$\mathrm{MinA}$ lower bound}&&\\
         \multirow{-3}{*}{\text{\Cref{minALL}}}&\multirow{-3}{*}{\text{$R^j>n(n-1)-\frac{1}{j}$}}&\text{in \Cref{conjLL}.}&\multirow{-3}{*}{\text{$\mathcal{VF}, \mathcal{F}$}}&\multirow{-3}{*}{\text{Yes}}\\
    \hline
         \rowcolor{white}
         &\text{$R^j>6-\frac{1}{j}$}& \text{Counterexample to}&&\\
          \rowcolor{white}
         \multirow{-2}{*}{\text{\Cref{minAMV}}}&\text{$\mathrm{width}(M_j)\geq 4\pi$}&\text{\Cref{conjnMV}.}&\multirow{-2}{*}{\text{$\mathcal{VF}, \mathcal{F}$}}&\multirow{-2}{*}{\text{Yes}}\\

     \hline
        && \text{No $\mathrm{GH}$-convergent}&&\\
        \multirow{-2}{*}{\text{\Cref{spikesLL}}}&\multirow{-2}{*}{\text{$R^j>n(n-1)-\frac{1}{j}$}}&\text{subsequence.}&\multirow{-2}{*}{\text{$\mathrm{VADB}, \mathcal{VF}, \mathcal{F}$}}&\multirow{-2}{*}{\text{?}}\\

    \hline
        \rowcolor{white}
         &\text{$R^j>6-\frac{1}{j}$}& \text{No $\mathrm{GH}$-convergent}&&\\
          \rowcolor{white}
         \multirow{-2}{*}{\text{\Cref{spikesMV}}}&\text{$\mathrm{width}(M_j)\geq 4\pi$}&\text{subsequence.}&\multirow{-2}{*}{\text{$\mathrm{VADB}, \mathcal{VF}, \mathcal{F}$}}&\multirow{-2}{*}{\text{?}}\\

      \hline
         \Cref{noncompact}&\text{$R^j>\kappa>0$}& \text{Not precompact.}& $\mathcal{VF}, \mathcal{F}$& {?} \\
        
    \hline
       \rowcolor{white}
         && \text{Generalized Scalar}&&\\
        \rowcolor{white}
         &\multirow{-2}{*}{\text{$R^j>\kappa$}}&\text{curvature is negative}&&\\
         \rowcolor{white}
         \multirow{-3}{*}{\Cref{sewingthm}}&\multirow{-2}{*}{\scriptsize\text{($R^j>0$, \cite{BDS})}}&\text{infinity at a point.}&\multirow{-3}{*}{\text{$\mathcal{F}$}}&\multirow{-3}{*}{\text{Yes}}\\
       
    \hline
         &\text{$R^j>\kappa>0$}& \text{No two points are}&&\\
         \multirow{-2}{*}{\text{\Cref{nogeothm}}}&\text{\scriptsize\text{($R^j>0$, \cite{BKS})}}&\text{connected by a geodesic.}&\multirow{-2}{*}{\text{$\mathcal{F}$}}&\multirow{-2}{*}{\text{Yes}}\\

    \hline
        \end{tabular}
    }
    
    \caption{Here we summarize the examples constructed in this paper.}
    \label{thmtable}
\end{table}

The paper is structured as follows. In \Cref{bg}, the background is discussed including some definitions and theorems related to different notions of convergence for Riemannian manifolds. In \Cref{constr}, we prove our main construction propositions: \Cref{propW} (\nameref{propW}) and \Cref{propT} (\nameref{propT}). In \Cref{theoremsA}, we use \Cref{propT} to prove Theorems \ref{minALL} and \ref{minAMV}. In \Cref{theoremsB}, we prove Theorems \ref{spikesLL},  \ref{spikesMV}, and \ref{noncompact} using \Cref{propW}. Finally, in Sections \ref{sewing} and \ref{nogeo}, we discuss how the construction propositions can be used to generalize the construction of sewing manifolds and the construction of sequences of smooth manifolds whose limit does not have any geodesics.

\section{Acknowledgements}
The author would like to thank Marcus Khuri and Raanan Schul for their invaluable guidance and encouragement throughout the process of producing this result. The author would also like to thank Christina Sormani for her helpful discussions and  the suggestion to construct examples related to Llarull's rigidity theorem. The author gratefully acknowledges support from the Simons Center for Geometry and Physics, Stony Brook University, at which some of the research for this paper was performed. This work was supported in part by NSF Grant DMS-2104229 and NSF Grant DMS-2154613.
\section{Background}\label{bg}

In this section, we will review different types of convergences between two Riemannian manifolds.

\subsection{Gromov-Hausdorff convergence} Here we will review the Gromov-Hausdorff distance between two metric spaces. Gromov  defined this distance between two metric spaces by generalizing the concept of Hausdorff distance between two subsets of a metric space. We refer the reader to \cite{Gr1} for further details.

The Gromov-Hausdorff distance between two metric spaces $(X_1,d_1)$ and $(X_2,d_2)$ is
\[
d_{\mathrm{GH}}((X_1,d_1),(X_2,d_2))=\inf_{Z} \{d_H^Z(\phi_1(X_1),\phi_2(X_2))\}
\]
where the infimum is taken over all complete metric spaces $(Z,d^Z)$ and all distance preserving maps $\phi_i:X_i\to Z $. We say that a metric spaces $(X_j,d_j)$ converge in the $\mathrm{GH}$-sense to a metric space $(X_\infty,d_\infty)$ if
\[
d_{\mathrm{GH}}((X_j,d_j),(X_\infty,d_\infty)) \to 0
\]
If, in addition, $\mu_j$ and $\mu_\infty$ are measures on $X_j$ and $X_\infty$, respectively, then Fukaya \cite{F} introduced the notion of metric measure convergence for metric measure spaces. We say $(X_j,d_j,\mu_j)$ converges to a metric measure space $(X_\infty,d_\infty,\mu_\infty)$ in metric measure ($\mathrm{mGH}$) sense if we have convergence in the $\mathrm{GH}$-sense and
\[
\phi_{j*}\mu_j\to \phi_{\infty *}\mu_\infty \text{ weakly as measures in } Z.
\]
We note that both define a distance between two Riemannian manifolds since there is a natural distance function and natural measure associated with a Riemannian manifold $(M,g)$. 

Gromov, in the following theorem, characterizes when a sequence of compact metric spaces contains a subsequence that converges in the $\mathrm{GH}$-sense.
\begin{theorem}\label{GH}
For a sequence of compact metric spaces $(X_j,d_j)$  such that $\diam{(X_j)}<D<\infty$, the following are equivalent:
\begin{enumerate}
    \item There exists a convergent subsequence.
    \item There is a function $N_1:(0,\alpha)\to (0,\infty)$ such that $Cap_j(\eps)\leq N_1(\eps)$
    \item  There is a function $N_2:(0,\alpha)\to (0,\infty)$ such that $Cov_j(\eps)\leq N_2(\eps)$,
    where 
    \[
    \text{Cap}_j(\eps)= \text{ maximum number of disjoint } \frac{\eps}{2}\text{-balls in } X_j,
    \]
     \[
    \text{Cov}_j(\eps)= \text{ minimum number of  } \eps\text{-balls it takes to cover } X_j.
    \]
\end{enumerate}
\end{theorem}

\subsection{Intrinsic Flat Convergence} In this section we will review Sormani-Wenger intrinsic flat distance between two integral current spaces. Sormani and Wenger \cite{SW} defined intrinsic flat distance, which generalizes the notion of flat distance for currents in Euclidean space. To do so they used  Ambrosio and Kirchheim's generalization of Federer and Fleming's integral currents to metric spaces. We refer the reader to \cite{AK} for further details about currents in arbitrary metric spaces and to \cite{SW} for further details about integral current spaces and intrinsic flat distance.

Let $(Z,d^Z)$ be a complete metric space. Denote by $\text{Lip}(Z)$ and $\text{Lip}_b(Z)$ the set of real-valued Lipschitz functions on $Z$ and the set of bounded real-valued Lipschitz functions on $Z$. 
\begin{deff}[\cite{AK}, Definition 3.1]
We say a multilinear functional 
\[
T:\text{Lip}_b(Z)\times [\text{Lip}(Z)]^m\to \bR
\]
on a complete metric space $(Z,d)$ is an $m$-dimensional current if it satisfies the following properties.
\begin{enumerate}
    \item Locality: $T(f,\pi_1,\ldots,\pi_m)=0$ if there exists and $i$ such that $\pi_i$ is constant on a neighborhood of $\{f\neq 0\}$.
    \item Continuity: $T$ is continuous with respect to pointwise convergence of $\pi_i$ such that $\text{Lip}(\pi_i)\leq 1$.
    \item Finite mass: there exists a finite Borel measure $\mu$ on $X$ such that 
    \begin{equation} \label{fm}
        |T(f,\pi_1,\ldots,\pi_m)|\leq \prod_{i=1}^m \text{Lip}(\pi_i) \int_Z |f| d\mu
    \end{equation}
    for any $(f,\pi_1,\ldots,\pi_m)$.
 \end{enumerate}
\end{deff}
We call the minimal measure satisfying (\ref{fm}) the mass measure of $T$ and denote it $||T||$. We can now define many concepts related to a current. $\mathbf{M}(T)=||T||(Z)$ is defined to be the mass of $T$ and the canonical set of a $m$-current $T$ on $Z$ is 
\[
\text{set}(T)=\left\{p\in Z \text{ }\Big|\text{ } \liminf_{r\to 0} \frac{||T||(B(p,r))}{r^m}>0\right\}.
\]
The boundary of a current $T$ is defined as $\partial T:\text{Lip}_b(X)\times [\text{Lip}(X)]^{m-1}\to \bR$, where
\[
\partial T(f,\pi_1,\ldots,\pi_{m-1})=T(1,f,\pi_1,\ldots,\pi_{m-1}).
\]
Given a Lipschitz map $\phi:Z\to Z'$, we can pushforward a current $T$ on $Z$ to a current $\phi_\# T$ on $Z'$ by defining
\[
\phi_{\#} T(f,\pi_1,\ldots,\pi_m) = T(f\circ \phi,f\circ \pi_1,\ldots,f\circ \pi_m).
\]
A standard example of an $m$-current on $Z$ is given by
\[
\phi_\#[[\theta]](f,\pi_1,\ldots,\pi_m) =\int_A (\theta \circ \phi) (f\circ \phi)d(\pi_1\circ \phi)\wedge \cdots \wedge d(\pi_m\circ \phi),
\]
where $\phi:\bR^m\to Z$ is bi-Lipschitz and $\theta\in L^1(A,\bZ)$.
We say that an $m$-current on $Z$ is integer rectifiable if there is a countable collection of bi-Lipschitz maps $\phi_i:A_i\to X$ where $A_i\subset \bR^m$ is precompact Borel measurable with pairwise disjoint images and weight functions $\theta_i\in L^1(A_i,\bZ)$ such that 
\[
T=\sum_{i=1}^\infty \phi_{i\#} [[\theta_i]]. 
\]
Moreover, we say an integer rectifiable current whose boundary is also integer rectifiable is an integral current. We denote the space of integral $m$-currents on $Z$ as $\mathbf{I}_m(Z)$. The flat distance between two integral currents $T_1$, $T_2\in\mathbf{I}(Z)$ is
\[
d^Z_F(T_1,T_2) = \inf\{\mathbf{M}(U)+\mathbf{M}(V)\mid U\in\mathbf{I}_m(X), V\in\mathbf{I}_{m+1}(X), T_2-T_1=U+\partial V\}.
\]

We say that the triple $(X,d,T)$ is an integral current space if $(X,d)$ is a metric space, $T\in\mathbf{I}_m(\bar{X})$ where $\bar{X}$ is the completion of $X$, and $\text{set}(T)=X$. The intrinsic flat $(\mathcal{F})$ distance between two integral current spaces $(X_1,d_1,T_1)$ and $(X_2,d_2,T_2)$ is
\[
d_\mathcal{F}((X_1,d_1,T_1),(X_2,d_2,T_2)) = \inf_Z\{d_F^Z(\phi_{1\#}T_1,\phi_{2\#}T_2)\}
\]
where the infimum is taken over all complete metric spaces $(Z,d^Z)$ and isometric embeddings $\phi_1:(\bar{X}_1,d_1)\to (Z,d^Z)$ and $\phi_2:(\bar{X}_2,d_2)\to (Z,d^Z)$. We note that if $(X_1,d_1,T_1)$ and $(X_2,d_2,T_2)$ are precompact integral current spaces such that 
\[
d_\mathcal{F}((X_1,d_1,T_1),(X_2,d_2,T_2))=0
\]
then there is a current preserving isometry between $(X_1,d_1,T_1)$ and $(X_2,d_2,T_2)$, i.e., there exists an isometry $f:X_1\to X_2$ whose extension $\bar{f}:\bar{X}_1\to \bar{X}_2$ pushes forward the current: $\bar{f}_\#T_1=T_2$. We say a sequence of $(X_j,d_j,T_j)$ precompact integral current spaces converges to $(X_\infty,d_\infty,T_\infty)$ in the $\mathcal{F}$-sense if
\[
d_\mathcal{F}((X_j,d_j,T_j),(X_\infty,d_\infty,T_\infty))\to 0.
\]
If, in addition, $\mathbf{M}(T_i)\to \mathbf{M}(T_\infty)$, then we say $(X_j,d_j,T_j)$ converges to $(X_\infty,d_\infty,T_\infty)$ in the voulme preserving intrinsic flat $(\mathcal{VF})$ sense. We note that we can view compact Riemannian manifolds $(M^n,g)$ as precompact integral current spaces $(M^n,d_g,\int_{M^n} dvol_g)$, where $d_g$ is the natural distance function on the Riemannian manifold and integration over the manifold, $\int_{M^n} dvol_g$, can be viewed as an integral current. Moreover, $\mathbf{M}(M^n)=\vol{(M^n)}$. Lakzian and Sormani in \cite{LaS} were able to estimate the intrinsic distance between two diffeomorphic manifolds:
\begin{theorem}\label{thm4.6}
Suppose $M^n_1=(M^n,g_1)$ and $M^n_2=(M^n,g_2)$ are oriented precompact Riemannian manifolds with diffeomorphic subregions $U_j\subset M^n_j$ and diffeomorphisms $\psi_j:U\to U_j$ such that for all $v\in TU$ we have
\[
\frac{1}{(1+\eps)^2} \psi^*_1g_1(v,v) < \psi^*_2g_2(v,v) <(1+\eps)^2\psi^*_1g_1(v,v).
\]
We define the following quantities 
\begin{enumerate}
    \item $ D_{U_j} = \sup \{\diam_{M_j}{(W)}: W \text{ is a component of } U_j\}$.
    \item Define $a$ to be a number such that $a>\frac{\arccos({1+\eps})^{-1}}{\pi} \max\{D_{U_1},D_{U_2}\}$.
    \item $\lambda = \sup_{x,y\in U}|d_{M_1}\left(\psi_1(x),\psi_1(y)\right) -d_{M_2}\left(\psi_2(x),\psi_2(y)\right)|$.
    \item $h=\sqrt{\lambda\left(\max\{D_{U_1},D_{U_2}\}+\frac{\lambda}{4}\right)}$.
    \item $\bar{h}=\max\left\{h,\sqrt{\eps^2+2\eps}D_{U_1},\sqrt{\eps^2+2\eps}D_{U_2}\right\}.$
\end{enumerate}

Then the intrinsic flat distance between $M^n_1$ and $M^n_2$ is bounded:
\begin{align*}
    d_\mathcal{F}(M_1,M_2) &\leq \left(2\bar{h}+a\right)\left(\vol_m(U_1) + \vol_m(U_2) + \vol_{m-1}(\partial U_1)+\vol_{m-1}(\partial U_2)\right)\\
    &\qquad + \vol_m(M_1\setminus U_1) +\vol_m(M_2\setminus U_2).
\end{align*}
\end{theorem}

Moreover, Sormani \cite{S3} proves the following Arzela-Ascoli theorem in the setting of $\mathcal{F}$-convergence.

\begin{theorem} \label{thm6.1}
Fix $L>0$. Suppose $M_j=(X_j,d_j,T_j)$ are integral current spaces for \\$j\in\{1,2,\ldots,\infty\}$ and $M_j\xrightarrow{\mathcal{F}}M_\infty$ and $F_j:X_j\to W$ are $L$-Lipschitz maps into a compact metric space $W$, then a subsequence converges to an $L$-Lipschitz map $F_\infty:X_\infty\to W$. Specifically, there exists isometric embeddings of the subsequence $\phi_j:X_j\to Z,$ such that $d_F^Z(\phi_{j\#}T_j,\phi_{\infty\#}T_\infty)\to 0$  and for any sequence $p_j\in X_j$ converging to $p\in X_\infty$,
\[
d_Z(\phi_j(p_j),\phi_\infty(p))\to 0,
\]
one has converging images
\[
d_W(F_j(p_j),F_\infty(p))\to 0.
\]
\end{theorem}

\subsection{Volume above distance below convergence} Allen, Perales, and Sormani in \cite{APS} introduced a new notion of convergence of manifolds called volume above distance below ($\mathrm{VADB}$) convergence. It is based on the volume-distance rigidity theorem which states that if there is a $C^1$-diffeomorphism $F:M\to N$ between two Riemannian manifolds which is also distance non-increasing then $\vol{(N)} \leq \vol{(M)}$; moreover, in case of equality the manifolds are isometric. 
\begin{deff}
A sequence of Riemannian manifolds without boundary $M^n_j=(M^n,g_j)$ converge in the $\mathrm{VADB}$-sense to a Riemannian manifold $M^n_\infty=(M^n,g_\infty)$ if 
\begin{enumerate}
    \item $\vol{(M^n_j)}\to \vol{(M^n_\infty)}$.
    \item $\diam{(M^n_j)}\leq D$.
    \item There exists a $C^1$-diffeomorphisms $\Psi_j:M^n_\infty \to M^n_j$ such that for all $p,q\in M^n_\infty$ we have 
    \[
    d_j(\Psi_j(p),\Psi_j(q)) \geq d_\infty (p,q).
    \]
\end{enumerate}
\end{deff}
We also record the following lemma from \cite{APS} which says that the above condition on the distance functions in the definition of $\mathrm{VADB}$-convergence can be converted into a condition on Riemannian metrics.
\begin{lemma}
Let $M^n_1=(M^n,g_1)$ and $M^n_0=(M^n,g_0)$ be Riemannian manifolds and $F:M^n_1\to M^n_0$ be a $C^1$-diffeomorphism. Then 
\[
g_0(dF(v),dF(v)) \leq g_1(v,v) \qquad \text{for all } v\in TM^n_1
\]
if and only if
\[
d_0(F(p),F(q)) \leq d_1(p,q) \qquad \text{for all } p,q\in M^n_1.
\]
\end{lemma}
Finally, we record the following theorem from \cite{APS} which describes the relationship between $\mathrm{VADB}$-convergence and $\mathcal{VF}$-convergence.
\begin{theorem}\label{VADB}
 If $M^n_j=(M^n,g_j)$ and $M^n_\infty=(M^n ,g_\infty)$ are compact oriented Riemannian manifolds such that $M^n_j \xrightarrow{\mathrm{VADB}} M^n_\infty$ then $M^n_j \xrightarrow{\mathcal{VF}} M^n_\infty$.
\end{theorem}

\section{Wells and Tunnels}\label{constr}
In this section, we prove the main new technical propositions: \Cref{propW} (\nameref{propW}) and \Cref{propT} (\nameref{propT}). These are an improvement of the constructions of Gromov-Lawson \cite{GL}, Basilio, Dodziuk, Sormani \cite{BDS}, and Dodziuk \cite{D}. We construct wells and tunnels and get control over the volume and diameter while keeping the scalar curvature close to the scalar curvature of the manifold to which we are attaching the well. \Cref{propW} (\nameref{propW}) allows us to remove a ball from a Riemannian manifold $M$ with scalar curvature $R^M\geq \kappa$ and glue in a well to create a new Riemannian manifold $N$; moreover, $M$ and $N$ will be isometric away from the gluing and the scalar curvature $R^N$ of $N$ will satisfy $R^N\geq \kappa-\eps$ for arbitrarily small $\eps$.  \Cref{propT} (\nameref{propT}) allows the analogous construction for connecting two manifolds with a tunnel. Therefore, given a Riemannian manifold $M$ with $R^M\geq \kappa$ we can remove two balls and glue in a tunnel to create a Riemannian manifold $P$ with $R^P\geq \kappa-\eps$ for arbitrarily small $\eps$.

\begin{prop}[Constructing Wells]\label{propW}
Let $(M^n,g)$, $n\geq 3$, be a Riemannian manifold with scalar curvature $R^M$. Let $\delta>0$ be small enough, $j\in \bN$, and $d>0$. If $R^M\geq \kappa$ on $B_g(p,2\delta)$ a ball in $(M,g)$, then we can construct a well $W_j=(B_g(p,2\delta),g_j)$ and a new complete Riemannian manifold $(N^n,h)$,
\[
N^n=M^n, \qquad h|_{M\setminus B_g(p,2\delta)}=g|_{M\setminus B_g(p,2\delta)},\qquad h|_{B_g(p,2\delta)}=g_j|_{B_g(p,2\delta)}.
\]
Furthermore, the following properties are satisfied:
\begin{enumerate}
    \item The scalar curvature, $R^{j}$, of ${W_j}$ satisfies $R^{j}>\kappa-\frac{1}{j}$.
    \item $ g_j|_E=g|_E$ where $E=B_g(p,2\delta)\setminus B_g(p,\delta)$ is identified with a subset of $W_j$.
    \item There exists constant $C>0$ independent of $j$ and $d$ such that
\[
\diam{({W_j})}<C(\delta +d) \quad \text{and} \quad \vol{({W_j})}<C(\delta^n+d\delta^{n-1}).
\]
\item N has scalar curvature $R^N>\kappa -\frac{1}{j}$.
\end{enumerate}
\end{prop}

 \begin{prop}[Constructing Tunnels]\label{propT}
Let $(M^n,g)$, $n\geq 3$, be a Riemannian manifold 

with scalar curvature $R^M$. Let $\delta>0$ be small enough, $j\in \bN$, and $d\geq 0$. If $R^M\geq \kappa$ on two balls $B_g(p,2\delta)$ and $B_g(p',2\delta)$ in $(M^n,g)$, then we can construct a new complete Riemannian manifold $P^n$, where we remove two balls and glue cylindrical region $(T_j,g_j)$ diffeomorphic to $\bS^{n-1}\times[0,1]$,
\[
P^n=M^n\setminus\left(B_g(p,2\delta)\cup B_g(p',2\delta)\right)\sqcup T_j.
\]

Furthermore, the following properties are satisfied:
\begin{enumerate}
    \item The scalar curvature, $R^{j}$, of ${T_j}$ satisfies $R^{j}>\kappa-\frac{1}{j}$.
    \item $g_j|_E=g|_E$ and $g_j|_{E'}=g|_{E'}$ where $E=B_g(p,2\delta)\setminus B_g(p,\delta)$ and $E'=B_g(p',2\delta)\setminus B_g(p',\delta)$ are identified with subsets of $P$.
    \item There exists constant $C>0$ independent of $j$ and $d$ such that
\[
\diam{({T_j})}<C(\delta +d) \quad \text{and} \quad \vol{({T_j})}<C(\delta^n+d\delta^{n-1}).
\]
\item $P$ has scalar curvature $R^P>\kappa -\frac{1}{j}$.
\end{enumerate}
\end{prop}

We adapt the proof from \cite{D}. The well and tunnel will be constructed as a codimension one submanifold. The submanifold will be defined by a curve, and this curve will control the geometry of the submanifold. First, we show how the curve defines the submanifold and how it affects its geometry. Second, we carefully construct the curve so that the submanifold will inherit the desired properties.

In particular, the construction will follow the following outline. First, we will describe how, given a curve, we can define a submanifold and write the scalar curvature in terms of quantities related to the curve. Second, we carefully construct a $C^1$-curve, $\gamma$, which will be used to define a submanifold that is the precursor to a well or a tunnel. Third, we adjust the construction of $\gamma$ so the resulting manifold will be a well. Fourth, we describe the smoothing procedure to make $\gamma$ a $C^\infty$-curve. Fifth, we construct a well and check it has the desired properties. Sixth, we perform the analogous steps to construct a tunnel with the desired properties.

\subsection{A Submanifold defined by a curve}
Let $(M^n,g)$ be a compact Riemannian manifold with scalar curvature $R^M\geq \kappa$. Let $\delta>0$ and $B=B(p,2\delta)$ be a geodesic ball in $M$. Consider the Riemannian product $(X,g_X)=(\bR \times B, dt^2+dr^2+g_r)$.  Let $\rho\in B$ be a geodesic radius from $p$ to $\partial B$ and define $S=\bR\times \rho$, which is a total geodesic submanifold of $\bR \times B$ with coordinates $(t,r)$. Let $\gamma$ be a smooth curve in $S$ to be determined later. Finally, let ${\Sigma} = \left\{(y,q)\in X : \left(y,||q||_g\right)\in \gamma \right\}$ be a submanifold of $(X,g_X)$ with the induced metric, where $||\cdot||_g$ is the distance from $p$ to $q$ with respect to $g$.
Now we want to calculate the scalar curvature of $\Sigma$. To do so we will need the following lemma from \cite{D}:

\begin{lemma}\label{princurv}
The principal curvatures of the hypersurface $\bS^{n-1}(\eps)$ in $B$ are each of the form $\frac{1}{-\eps} + O(\eps)$ for small $\eps$. Furthermore, let $g_\eps$ be the induced metric on $\bS^{n-1}(\eps)$ and let $g_{rd,\eps}$ be the round metric of curvature $\frac{1}{\eps^2}$. Then, as $\eps \to 0$, $\frac{1}{\eps^2}g_\eps \to \frac{1}{\eps^2}g_{rd,\eps} =g_{rd}$ in the $C^2$ topology, moreover,  $||g_{rd}-\frac{1}{\eps^2}g_\eps||\leq\eps^2$.
\end{lemma}

 Now to calculate the scalar curvature of $\Sigma$, fix $q \in \Sigma \cap S$. Let $e_1,\dots,e_n$ be an orthonormal basis of of $T_q({\Sigma})$ where $e_1$ is tangent to $\gamma$. Note that the for points in ${\Sigma}\cap S$ the normal $\nu$ to ${W}$ in $X$ is the same as the normal to $\gamma$ in $S$. 

From the Gauss equations:

\[
R^X(X,Y,Z,U)=R^{\Sigma}(X,Y,Z,U) - A(X,U)A(Y,Z) + A(X,Z)A(Y,U)
\]
we see 
\[
K^{\Sigma}_{ij}=K^X_{ij} +\lambda_i\lambda_j.
\]
where $\lambda_i$ are principal curvatures corresponding to $e_i$ and $K^{\Sigma}_{ij}$ and $K^X_{ij}$ are the respective sectional curvatures. We note that $\lambda_1 = k$ where $k$ is the geodesic curvature of $\gamma$. For $i=2,\ldots,n$ we see by \Cref{princurv}
\[
\begin{split}
\lambda_i &= \langle\nabla_{\partial_{i}} \nu, \partial_{i}\rangle \\
&= \langle\nabla_{\partial_{i}} \cos \theta \partial_t+\sin\theta \partial_r, \partial_{i}\rangle\\
&=\cos \theta \langle\nabla_{\partial_{i}} \partial_t,\partial_i\rangle+\sin \theta \langle\nabla_{\partial_{i}} \partial_r, \partial_{i}\rangle\\
&= \sin \theta \langle\nabla_{\partial_{i}} \partial_r,\partial_i\rangle\\
&=\left(\frac{1}{-r} + O(r) \right) \sin \theta,
\end{split}
\]
where $\theta$ is the angle that between $\nu$ and the $t$-axis.
Now note that
\[
K^X_{1j}=R^X(e_j,e_1,e_1,e_j)=R^X(e_j,\cos \theta \partial_r,\cos \theta \partial_r,e_j)=\cos^2\theta K^M_{\partial_r, j}.
\]
For $i\neq 1$ and $j\neq 1$
\[
K^X_{ij}=R^X(e_j,e_i,e_i,e_j)=\ K^M_{i, j}.
\]
Since 
\[
R^{{\Sigma}} = \sum_{i\neq j} K^{{\Sigma}}_{ij}
\]
we see
\begin{equation}\label{scal}
    \begin{split}
        R^{{\Sigma}} &= R^M - 2\text{Ric}^M\left(\partial_r,\partial_r\right)\sin^2\theta \\
                            &+ (n-2)(n-1)\left(\frac{1}{r^2} + O(1)\right) \sin^2\theta \\
                            & - (n-1)\left(\frac{1}{r} + O(r)\right) k \sin\theta.
    \end{split}
\end{equation}

\subsection{Constructing the Curve}
The construction of the curve that will define the well $W$ and the construction of the curve that will define the tunnel $T$ are very similar. First, we will construct a curve that will define a submanifold $\Sigma$, which can be thought of as the precursor to a well or a tunnel. 
 
We want to construct a curve $\gamma$ so that the resulting manifold $\Sigma$ has $R^{\Sigma}>\kappa-\frac{1}{j}$ for any $j\in\bN$.
We will first construct $\gamma$ as a piecewise curve of circular arcs and then smooth the curve. To do this, we will prescribe the geodesic curvature $k(s)$ of $\gamma$, and by Theorem 6.7 in \cite{G}, we know that $k(s)$ determines $\gamma$. The unit tangent vector to $\gamma$ and the curvature are given by
\[
\frac{d\gamma}{ds} = (\sin\theta,-\cos\theta) \qquad \text{and} \qquad k =\frac{d\theta}{ds}.
\]
Therefore, if $\gamma(s)$ is defined for $s\leq s'$ and $k(s)$ is given for $s\geq s'$ we have $\gamma(s)=(t(s),r(s))$ where

\begin{equation} \label{curve}
    \begin{split}
        &\theta(s) = \theta(s') +\int_{s'}^s k(u)du\\
        &t(s) = t(s') +\int_{s'}^s \sin\theta(u) du\\
        &r(s) = r(s') - \int_{s'}^s \cos\theta(u) du.\\
   \end{split}
\end{equation}

Now, we begin the construction of $\gamma$. Fix $j\in\bN$. Let $\delta_0<\delta$ and let $(0,\delta_0)$ be a point in the $(t,r)$-plane. Next, define the initial segment of $\gamma$ as the line segment from $(0,2\delta)$ to $(0,\delta_0)$ for $s\in[-2\delta,0]$. Define the next segment to be an arc of a circle of curvature $k_0=1$ that is tangent to $r$-axis at $(0,\delta_0)$ and let $\gamma$ run from $0$ to $s_0 \leq \frac{\delta_0}{2}$ where $s_0$ is chosen so that $R^{\Sigma}>\kappa-\frac{1}{j}$ and that {$\frac{\sin\theta(s_0)}{8r(s_0)}<1$} for all $s\leq s_0$. We note that $s_0$ exists since $\theta(0)=0$ and by the scalar curvature formula (\ref{scal}). Next, we prove a lemma that gives a condition on $\gamma$ that controls the scalar curvature. 

\begin{lemma}\label{scalposlem}
If $\delta_0$ is small enough and if 
\begin{equation}\label{scalpos}
    \frac{\sin \theta(s)}{4r(s)} > k(s) \text{ for } s\geq s_0,
\end{equation}
then $R^{\Sigma}>\kappa$.
\end{lemma}
\begin{proof}
By (\ref{scal}) we see if $k\leq 0$ then
\begin{equation}
    \begin{split}
        R^{\Sigma} &= R^M - 2\text{Ric}^M\left(\partial_r,\partial_r\right)\sin^2\theta \\
                            &+ (n-2)(n-1)\left(\frac{1}{r^2} + O(1)\right) \sin^2\theta \\
                            & - (n-1)\left(\frac{1}{r} + O(r)\right) k \sin\theta.
     \end{split}
\end{equation}
and so the third and fourth terms will be nonnegative. By taking $\delta_0>r$ small enough, the third and fourth terms will dominate the second term so $R^{\Sigma}>\kappa$.

Now, if $k>0$, then by rewriting the right-hand side of (\ref{scal}) we get
\begin{equation}
    \begin{split}
        R^{\Sigma} &= \frac{(n-2)(n-1)}{2r^2}\sin^2\theta + \left(\frac{(n-2)(n-1)}{2r^2}- 2\text{Ric}^M\left(\partial_r,\partial_r\right)+O(1)\right)\sin^2\theta\\
        &+\frac{-2(n-1)k}{r}\sin\theta +\left(\frac{(n-1)}{r}-O(r)\right)k\sin\theta\\
        &+R^M,
    \end{split}
\end{equation}
so second and fourth terms will be positive by taking $\delta_0>r$ is small enough and by assumption we have
\[
 \frac{\sin \theta}{4r} > k
 \]
 which implies
\[
\frac{(n-2)(n-1)}{2r^2}\sin^2\theta +\frac{-2(n-1)k}{r}\sin\theta >0,
\]
and so $R^\Sigma>\kappa.$
\end{proof}
 Thus, as we continue to construct $\gamma$, we will ensure that (\ref{scalpos}) is satisfied. We will now extend $\gamma$ by a circular arc of curvature $k_1=\frac{\sin\theta(s_0)}{8r(s_0)}$ on $[s_0,s_1]$ where $s_1-s_0 =\frac{r_0}{2}$, where $r(s_0)=r_0$. Let $\theta(s_0)=\theta_0$. By (\ref{curve}), we have first that $\sin\theta(s)$ is increasing and $r(s)$ is decreasing and so on $[s_0,s_1]$
\[
\frac{\sin\theta(s)}{4r(s)}> \frac{\sin\theta_0}{4r_0} >\frac{\sin\theta_0}{8r_0} =k_1.
\]
Second, we see that  $\gamma$ does not cross the $t$-axis because $s_1-s_0 =\frac{r_0}{2}$, and third we have
\[
\theta(s_1)-\theta_0 =  k_1(s_1-s_0) =\frac{\sin \theta_0}{8r_0} \frac{r_0}{2} = \frac{\sin \theta_0}{16}.
\]

Now we proceed inductively. Define:
\[
s_i=s_{i-1}+\Delta s_i, \qquad \Delta s_i =\frac{r_{i-1}}{2}, \qquad r_i=r(s_i), \qquad \theta_i=\theta(s_i), \qquad k_i =\frac{\sin \theta_{i-1}}{8r_{i-1}}.
\]

As $\theta(s)$ is increasing we have that $\theta_i - \theta_{i-1} = \frac{\sin \theta_{i-1}}{16}> \frac{\sin \theta_{0}}{16}$ and so
\[
\theta_i \geq \theta_0 +i\frac{\sin\theta_0}{16}.
\]
Therefore, $\theta_i$ grows without bound so define $m$ to be such that $\theta_{m-1}<\sin^{-1}\left(\frac{12}{13}\right)\leq \theta_m$. Redefine $s_m$ so that $\theta_m:= \sin^{-1}\left(\frac{12}{13}\right) =\bar{\theta}$. Note that $\Delta s_m\leq \frac{r_{m-1}}{2}$.

Now extend again by one circular arc. To do this we need to define $k_{m+1}>0$ and $s_{m+1}=s_m+\Delta  s_{m+1}$. We add a circular arc until $\theta_{m+1}=\frac{\pi}{2}$. By the definition of $\bar{\theta}$, there exists a $k_{m+1}$ such that $1-\sin\bar{\theta} < k_{m+1}\frac{r_m}{2}< \frac{\sin\bar{\theta}}{8}$ and by (\ref{curve}) we know
\begin{align*}
    r_{m+1} &= r_m - \int_{s_m}^{s_{m+1}} \cos \theta(u) du\\
    &=r_m- \int_{s_m}^{s_{m+1}} \cos\left(s_m+k_{m+1}(u-s_m)\right)du\\
    &=r_m - \frac{1}{k_{m+1}} \left(\sin\theta_{m+1} -\sin \theta_m\right)\\
    &= r_m - \frac{1}{k_{m+1}} \left(1-\sin\bar{\theta}\right)\\
    &>\frac{r_m}{2}.
\end{align*}
and $k_{m+1}<\frac{\sin\bar{\theta}}{4r_m}$. 

\begin{remark}\label{rmkT}
Up until this point the curve $\gamma$ works for both the construction of a well and a tunnel. However, from here on the construction of $\gamma$ differs slightly for the well and the tunnel. We will continue now with the construction of the well and discuss the tunnel construction later in Subsection \ref{subT}.
\end{remark}

\subsection{Adjusting the curve to construct a well}
Now we will refine our construction of $\gamma$ in order to construct a well. We want to extend by a line with a negative slope of length $d>0$ and not have $\gamma$ cross the $t$-axis. By the intermediate value theorem
there exists an $\hat{s}\in (s_{m},s_{m+1})$ such that $\theta(\hat{s})= \widehat{\theta}$ where
\begin{equation} \label{stopangle}
\begin{cases}
\max\left\{\bar{\theta}, \,\cos^{-1}\left(\frac{r_{m+1}}{2}\right)\right\} < \widehat{\theta} < \frac{\pi}{2} & \text{ if } d\leq 1 \\
\max\left\{\bar{\theta}, \,\cos^{-1}\left(\frac{r_{m+1}}{2d}\right)\right\} < \widehat{\theta} < \frac{\pi}{2} & \text{ if } d>1 \\
\end{cases}
\end{equation}
since $\theta_{m+1}=\frac{\pi}{2}$.

Redefine $s_{m+1}$ such that $s_{m+1}=\hat{s}$ and $\theta_{m+1}=\widehat{\theta}$. Extend $\gamma$ to $[s_{m+1},s_{m+1}+d]$ by setting $k=0$ on $[s_{m+1},s_{m+1}+d]$. Furthermore, note that by (\ref{curve}) we have $\theta(u)\equiv\theta_{m+1}$ on that interval and
\begin{align*}
    r(s_{m+1}+d) &= r_{m+1} - \int_{r_{m+1}}^{r_{m+1}+d} \cos\theta(u)du \\
    &=r_{m+1} -  d\cos\theta_{m+1} \\
    &\geq r_{m+1} - \frac{r_{m+1}}{2} \\
    &>0.
\end{align*}

Let $s_{m+1}+d =s_{m+2}$ and $\theta(s_{m+1}+d)=\theta_{m+2}$.  We now extend on $[s_{m+2},s_{m+3}]$ by a small circular arc of negative geodesic curvature such that $\theta(s_{m+3})=0$. Take 
\[
k_{m+3} <\frac{-2\sin\theta_{m+2}}{r_{m+2}}.
\]
Since,

\begin{align*}
    \theta(s) =\theta_{m+2} + \int_{s_{m+2}}^s k_{m+3} du = \theta_{m+2} + k_{m+3}(s-s_{m+2}).
\end{align*}
we have
\begin{align*}
    r(s_{m+3}) &=r_{m+2} - \int_{s_{m+2}}^{s_{m+3}} \cos\theta(u)du\\
    r_{m+3} &= r_{m+2} - \frac{1}{k_{m+3}}\left(\sin\theta_{m+3} - \sin\theta_{m+2}\right)\\
    r_{m+3} &= r_{m+2} +  \frac{1}{k_{m+3}}\sin\theta_{m+2}\\
    &>0.
\end{align*}

We can extend $\gamma$ on $[s_{m+3},s_{m+4}]$ by a vertical straight line by setting $k_{m+4} =0$, where $s_{m+4}$ is chosen so that $r(s_{m+4})=0$. 

Since $\gamma$ is parameterized by arclength, we note that a bound on $s_{m+4}$ is a bound on arclength. In following lemmas, we prove an upper bound for $s_{m+4}$.
\begin{lemma}\label{length1}
There exists a constant $0<C_1<1$ independent of $j$ and $d$ such that
\[
\frac{r_i}{r_{i-1}}\leq C_1
\]
for $1\leq i\leq m-1$
\begin{proof}
By (\ref{curve}) and by the mean value theorem we have
\[
r_i=r_{i-1} -\int_{s_{i-1}}^{s_i} \cos\theta(u)du = r_{i-1} - \Delta s_i \cos\xi_i = r_{i-1}\left(1-\frac{\cos\xi_i}{2}\right)
\]
for some $\xi_i\in[s_{i-1},s_i].$ Recalling that $\bar{\theta}\geq \xi_i$ for $1\leq i\leq m-1$ we see that
\[
\frac{r_i}{r_{i-1}}\leq 1-\frac{\cos\bar{\theta}}{2} = \frac{21}{26}.
\]

\end{proof}
\end{lemma}
\begin{lemma} \label{length2}
There is a constant $C_2$ independent of $j$ and $d$ such that
$s_{m+4}\leq C_2\delta_0+d$ which implies that the length of $\gamma$ is bounded by $C_2\delta_0+d.$
\end{lemma}
\begin{proof}
We recall for $1\leq i \leq m$, $\Delta s_i \leq \frac{r_{i-1}}{2}$ so
\begin{equation}\label{len}
\begin{split}
    s_m&=2\delta-\delta_0+s_0+\Delta s_1 +\cdots +\Delta s_m\\
    &\leq 2\delta+s_0 + \frac{1}{2} \left(r_0+r_1+\cdots+r_{m-1}\right)\\
    & \leq 2\delta+s_0 + \frac{r_0}{2} \left(1+C_1+\cdots +C_1^{m-1}\right)\\
    &\leq 3\delta+\frac{\delta_0}{2}\left( \frac{1}{1-C_1}\right)\\
    &\leq \frac{28}{5}\delta
\end{split}
\end{equation}

 Now, we note that by (\ref{curve}) and (\ref{stopangle}):
\[
\Delta s_{m+1} \leq \frac{1}{k_{m+1}} \left(\frac{\pi}{2}-\bar{\theta}\right) < r_m \frac{\frac{\pi}{2}-\bar{\theta}}{1-\sin\bar{\theta}} \leq  \frac{\frac{\pi}{2}-\bar{\theta}}{1-\sin\bar{\theta}} \delta_0 = 13\left(\frac{\pi}{2} - \sin^{-1}\left(\frac{12}{13}\right)\right)\delta_0 .
\]
By (\ref{curve}), we have that
\[
\theta_{m+3}= \theta_{m+2} +k_{m+3}\Delta s_{m+3}.
\]
Therefore,
\begin{equation}
    \Delta s_{m+3} = \frac{\theta_{m+2}}{-k_{m+3}} \leq \frac{2r_{m+2}\theta_{m+2}}{\sin\theta_{m+2}} \leq \frac{\pi}{\sin\bar{\theta}} \delta_0 = \frac{13\pi}{12} \delta_0
\end{equation}
because $\theta_{m+2}\leq \frac{\pi}{2}$, $r_{m+2} < \delta_0$, and $\bar{\theta} < \theta_{m+2}$.

By construction,
\[
\Delta s_{m+2} = d \qquad \text{ and } \qquad \Delta s_{m+4} \leq \delta_0.
\]
Thus,
\begin{align*}
  s_{m+4} &=s_0+\Delta s_1 +\cdots +\Delta s_m + \Delta s_{m+1} + \Delta s_{m+2} + \Delta s_{m+3} + \Delta s_{m+4}\\
  &\leq C_2\delta+d.
\end{align*}
\end{proof}

\subsection{Smoothing the curve that defines the well}
So far we have constructed $k(s)$ as a piecewise constant function, $k\big|_{(s_i,s_{i+1}]} = k_{i+1}$
The resulting curve $\gamma$ is $C^1$ and piecewise $C^\infty$. 

We begin the smoothing of $\gamma$ by first smoothing out $k(s)$ on $[0, s_{m+3}]$. Let $g\in C^\infty(\bR)$ be a smooth function so that
$g$ is $0$ if $s<0$, $1$ if $s>1$, and strictly increasing on $[0,1]$. Let $h(x)=g(1-x)$ and $H=\int_0^1 h(x) dx$. Let $\tilde{k}(s)$ be the smooth function defined by 
\[
\widetilde{k}(s)=\begin{cases}
g\left(\frac{s}{\alpha}\right) & s\in\left[-\frac{\delta_0}{2},\alpha\right]\\
1& s\in [\alpha,s_{0}-\alpha]\\
(1-k_1)h\left(\frac{s-s_0}{\alpha}\right) +k_1 &s\in[s_{0}-\alpha,s_0]\\
k_1& s\in[s_0,s_1]\\
\left(k_{i+1}-k_i\right)g\left(\frac{s-s_i}{\alpha}\right) + k_i & s\in[s_i,s_i+\alpha]\\
k_{i+1}& s\in [s_i+\alpha,s_{i+1}]\\
k_{m+1}h\left(\frac{s-s_{m+1}}{\alpha}\right)  & s\in[s_{m+1},s_{m+1}+\alpha]\\
0& s\in [s_{m+1}+\alpha,s_{m+2}]\\
-k_{m+3}h\left(\frac{s-s_{m+2}}{\alpha}\right)+k_{m+3}  & s\in[s_{m+2},s_{m+2}+\alpha]\\
k_{m+3}& s\in [s_{m+2}+\alpha,s_{m+3}],
\end{cases}
\]
where $1\leq i\leq m$.

We note that $\alpha$ is the same for each $i$ and that its value will be determined later. Now let $\tilde{\theta}$ be the angle function associated to $\tilde{k}$.  By (\ref{curve}) we see that the smooth curve $\tilde{\gamma}(s)=(\tilde{t}(s),\tilde{r}(s))$ defined by $\tilde{k}(s)$  will converge uniformly to $\gamma$ on ${\left[\frac{\delta_0}{2},s_{m+3}\right]}$ as $\alpha$ goes to zero. Also, $\tilde{\theta}$ will converge uniformly to $\theta$ as $\alpha$ goes to zero.  Therefore, take $\alpha$ small enough such that $\tilde{\theta}(s_{m+1})$ satisfies (\ref{stopangle}); therefore, we will still extend by a line with a negative slope.

Note
\begin{align*}
\tilde{\theta}(s_{m+2}+\alpha) &= \tilde{\theta}(s_{m+2}) + \int_{s_{m+2}}^{s_{m+2}+\alpha} -k_{m+3}h\left(\frac{u-s_{m+2}}{\alpha}\right)+k_{m+3}  du\\
&=\tilde{\theta}(s_{m+2}) + \alpha k_{m+3}(1-H) \\
&> 0.
\end{align*}
By the smoothing process, $\tilde{\theta}(s_{m+3})$ may no longer be greater than 0. We will now fix that. If $\tilde{\theta}(s_{m+3}) \leq 0$  pick a $s^*\in (s_{m+2}+\alpha,s_{m+3}]$ such that $0<\tilde{\theta}(s^*)<\alpha$ which exists by the intermediate value theorem. 

If $\tilde{\theta}(s_{m+3}) >0$, we can redefine $s_{m+3}$ as $s_{m+3}+\frac{\tilde{\theta}(s_{m+2})}{-k_{m+3}}$ so that $\tilde{\theta}(s_{m+3}) =0$. By the intermediate value theorem, pick a $s^*\in (s_{m+2}+\alpha,s_{m+3}]$ such that $0<\tilde{\theta}(s^*)<\alpha$. 

Redefine $s_{m+3}$ in either case as $s_{m+3}=s^*$ and note $0<\tilde{\theta}(s^*)<\alpha$. On $[s_{m+3},s_{m+3}+2\beta]$ define
\begin{align*}
    \tilde{k}(s) = \begin{cases}
    -k_{m+3}g\left(\frac{s-s_{m+3}}{\beta}\right)+k_{m+3} & s\in [s_{m+3},s_{m+3}+\beta]\\
    0 & s\in [s_{m+3}+\beta,s_{m+3}+2\beta],
    \end{cases}
\end{align*}
where $\beta = \frac{\tilde{\theta}(s_{m+3})}{-k_{m+3}(1-H)}$ so that
\[
\int_{s_{m+3}}^{s_{m+3}+\beta} \tilde{k}(s) ds =-\tilde{\theta}(s_{m+3}).
\]
This makes $\tilde{\theta}(s_{m+3}+2\beta)=0.$

By (\ref{curve}) we see that the smooth curve $\tilde{\gamma}(s)=(\tilde{t}(s),\tilde{r}(s))$ defined by $\tilde{k}(s)$  will converge uniformly to $\gamma$ on ${\left[-2\delta,s_{m+3}\right]}$ as $\alpha$ goes to zero. 

Also, $\tilde{\theta}$ will converge uniformly to $\theta$ as $\alpha$ goes to zero; moreover, as $\alpha$ goes to zero so does $\beta$. Finally, take $\alpha$ small enough so that $\tilde{r}(s_{m+3}+2\beta)>0$. Extend the line segment at the end of $\tilde{\gamma}$ on $[s_{m+3}+2\beta,L]$ where $L$ is defined so that $\tilde{r}(L)=0$. Note that $|L-(s_{m+3}+2\beta)|< \delta_0$.

\subsection{Attaching the Well} We have constructed a smooth curve $\tilde{\gamma}$ on  ${\left[-2\delta,L\right]}$ that begins and ends as a vertical line segment. Define 
\[
{\tilde{W}_j}= \left\{(y,q)\in X : \left(y,||q||_M\right)\in \tilde{\gamma} \right\}
\]
and let $g_j$ be the induced metric, i.e., $\tilde{g}_j=\tilde{\iota}_j^*(dt^2+g)$ where $\tilde{\iota}_j:\bar{W}_j\to \bR \times B$ is the inclusion map.

\begin{lemma}\label{scalandlen}
For small enough $\alpha$ we will show that $\tilde{\gamma}$ satisfies $R^{\tilde{W}_j}\geq\kappa-\frac{1}{j}$ on $\left[-2\delta,L\right]$. Also the length  $\tilde{\gamma}$ is bounded by $C_3\delta+d$.
\end{lemma}
\begin{proof}
By construction $\tilde{\gamma}$ is parameterized by arclength so by \Cref{length2} length of  $\tilde{\gamma}$ is bounded by $C_3\delta_0+d$. 

On $\left[-2\delta,0\right]$, we have that 
\begin{align*}
     R^{{\tilde{W}_j}} &= R^M - 2\text{Ric}^M\left(\partial_{\tilde{r}},\partial_{\tilde{r}}\right)\sin^2\tilde{\theta}  + (n-2)(n-1)\left(\frac{1}{\tilde{r}^2} + O(1)\right) \sin^2\tilde{\theta}  \\&\quad\quad- (n-1)\left(\frac{1}{\tilde{r}} + O(r)\right) \tilde{k} \sin\tilde{\theta}\\
     &= R^M \\
     &>\kappa-\frac{1}{j}.
\end{align*}

On $[0,s_0]$, we have that $k_1<1$ because of our choice of $s_0$ and the construction.
\begin{align*}
     R^{{\tilde{W}_j}} &= R^M - 2\text{Ric}^M\left(\partial_{\tilde{r}},\partial_{\tilde{r}}\right)\sin^2\tilde{\theta}  + (n-2)(n-1)\left(\frac{1}{\tilde{r}^2} + O(1)\right) \sin^2\tilde{\theta}  \\&\quad\quad- (n-1)\left(\frac{1}{\tilde{r}} + O(r)\right) \tilde{k} \sin\tilde{\theta}\\
     &\geq \kappa - 2\text{Ric}^M\left(\partial_{\tilde{r}},\partial_{\tilde{r}}\right)\sin^2\tilde{\theta}  + (n-2)(n-1)\left(\frac{1}{\tilde{r}^2} + O(1)\right) \sin^2\tilde{\theta}  \\&\quad\quad- (n-1)\left(\frac{1}{\tilde{r}} + O(r)\right)  \sin\tilde{\theta} \\
     &>\kappa-\frac{1}{j}.
\end{align*}
since for small enough $\alpha$ we have that $\tilde{\theta}$ is uniformly close to $\theta$ and $\tilde{r}$ is uniformly close to $r$.

On $[s_0,s_1]$ we have that $\tilde{k}(s)=k_1$ and so 
\[
\frac{\sin \tilde{\theta}(s)}{4\tilde{r}(s)} - \tilde{k}(s) >0
\]
for small enough $\alpha$.

On $[s_i,s_{i+1}]$ for $1\leq i\leq m$, we have that
\[
\frac{\sin \tilde{\theta}(s)}{4\tilde{r}(s)} - \tilde{k}(s) = \left(\frac{\sin \tilde{\theta}(s)}{4\tilde{r}(s)}-k_{i+1}\right) + \left(k_{i+1}-\tilde{k}(s)\right)
\]
and so for small enough $\alpha$ we have the the first term is positive since $\frac{\sin {\theta}(s)}{4{r}(s)} > {k}(s)$ and the second term is positive by construction. 

On $[s_{m+1},s_{m+2}]$, we have that
\[
\frac{\sin \tilde{\theta}(s)}{4\tilde{r}(s)} \geq \frac{\sin \tilde{\theta}(s_{m+1})}{4\tilde{r}(s_{m+1})}>k_{m+1} > \tilde{k}(s).
\]
We have the first inequality since $\frac{\sin \tilde{\theta}(s)}{4\tilde{r}(s)}$ is non-decreasing. The second inequality was already verified above. The third inequality holds since by construction $k_{m+1} > \tilde{k}(s)$.

On $[s_{m+2},L]$, we have by construction that $\tilde{k}(s)$ is non-positive so 
\[
\frac{\sin \tilde{\theta}(s)}{4\tilde{r}(s)} \geq \tilde{k}(s).
\] 
Therefore, by \Cref{scalposlem} we have shown $R^{{\tilde{W}_j}}>\kappa-\frac{1}{j}$ on $\left[-\frac{\delta_0}{2},L\right]$.
\end{proof}

Next we will prove the diameter and volume bounds for the well, but before we prove those bounds, we need to recall the following fact.
\begin{prop}\label{propvol}
Let $B(p,r)$ be a geodesic ball of radius $r$ in a closed Riemannian manifold $(M^n,g)$. Then there exists constants $C,r_0$ depending on $g$ such that for any $p$ and for all $r\leq r_0$ we have
\begin{equation}\label{volcomp}
  \begin{split}
  &   \vol_g(B(p,r)) \leq C r^n\\
    &  \vol_g(\partial B(p,r)) \leq C r^{n-1}. 
\end{split}
\end{equation}
\end{prop}

\begin{lemma}\label{diamandvol}
There is a constant $C(g)$ independent of $j,d$
such that diameter $\diam{({{\tilde{W}_j}})}$ and volume $\vol{({{\tilde{W}_j}})}$ of ${{\tilde{W}_j}}$ satisfy
\[
d\leq \diam{({{\tilde{W}_j}})} < C(\delta+d) \text{  and  } \vol{({{\tilde{W}_j}})} < C(\delta^n + d\delta^{n-1}).
\]
\end{lemma}
\begin{proof}
Let $p,q\in {\tilde{W}_j}$ be two points and let $x$ be the point at the tip of ${\tilde{W}_j}$, i.e., corresponding to $\tilde{\gamma}(L)$. 
By the triangle inequality and \Cref{princurv} we have
\[
d_{g_j}(p,q) \leq d_{g_j}(p,x) + d_{g_j}(x,q) \leq \text{length}(\tilde{\gamma}) + \text{length}(\tilde{\gamma}) \leq C(\delta+d).
\]
By construction we have $d\leq\diam{({W})}$. Therefore,
$d\leq\diam{({W})} < C(\delta+d).$

By possibly taking $\delta$ smaller, we have by \Cref{scalandlen} and \Cref{propvol} that
\[
\vol{({\tilde{W}_j})}  = \int_{\frac{-\delta_0}{2}}^L |\partial B(p,r(s))|_{\tilde{g}_j} ds \leq \int_{\frac{-\delta_0}{2}}^L C\delta^{n-1} ds \leq C(\delta^n+\delta^{n-1}d).
\]
\end{proof}
\begin{lemma}\label{smoothattach}
$({\tilde{W}_j},\tilde{g}_j)$ is isometric to $W_j=(B_g(p,2\delta),g_j=dF^2_j+g)$ and $W_j$ attaches smoothly to $M$.
\end{lemma}
\begin{proof}
Let $B=B_g(p,2\delta)$ and recall that $||q||_g$ is the distance from $q$ to $p$ in $B$. Consider the function $F_j:B\to \bR$, $F_j(q)=\tilde{t}\left(\tilde{r}^{-1}(||q||_g)\right)$. By construction $\tilde{t}$ is smooth and $\tilde{r}'(s) <0$ so $\tilde{r}^{-1}$ is smooth. Moreover, $||q||$ is smooth away from  $p$. Thus, away from $p$, $F$ is smooth. In a neighborhood of $p$ we have by construction that $(\tilde{t}(s),\tilde{r}(s))$ is a vertical line segment so in that neighborhood $\tilde{t}\circ\tilde{r}^{-1}\equiv const$ and so $F_j$ is smooth everywhere. Furthermore, by construction, we have that
\[
 \tilde{g}_j|_E=g|_E \text{ where } E=B_g(p,2\delta)\setminus B_g(p,\delta).
 \]
 
 Let $\Gamma_j=\{(t,p)\in X:F_j(p)=t\}$. Note that $\Gamma_j\subset X$ and that $\Gamma_j={\tilde{W}_j}$. Let $g'_j=(\iota'_j)^*(dt^2+g)$ where $\iota'_j:{\tilde{W}_j}\to \bR \times B$ is the inclusion map $\iota'_j(t,p)=(t,p)$. Let $id_j:\Gamma_j\to {\tilde{W}_j}$ be the identity map and conclude that $g'_j=\tilde{g}_j$. Consider the diffeomorphism $\Phi_j:B\to \Gamma_j$ where $\Phi_j(q)\mapsto (F_j(q),q)$. And so 
\[
\Phi_j^*g'_j = \Phi^*((\iota'_j)^*g'_j) = dF^2_j + g.
\]
\end{proof}
And this completes the construction of $N$ from \Cref{propW} (Constructing Wells). 

\subsection{Constructing a Tunnel} \label{subT}
We will pick up the construction of the tunnel from \Cref{rmkT}. Let $\gamma$ be as it is before \Cref{rmkT}. The same smoothing procedure as above can be used to smooth $\gamma$ into a smooth curve. We will abuse notation and call this smoothed-out curve $\tilde{\gamma}$ as well.

Let $g\in C^\infty(\bR)$ be the smooth function so that
$g$ is $0$ if $s<0$, $1$ if $s>1$, and strictly increasing on $[0,1]$. Let $h(x)=g(1-x)$ and $H=\int_0^1 h(x) dx$. Let $\tilde{k}(s)$ be the smooth function defined by 
\[
\widetilde{k}(s)=\begin{cases}
g\left(\frac{s}{\alpha}\right) & s\in\left[-\frac{\delta_0}{2},\alpha\right]\\
1& s\in [\alpha,s_{0}-\alpha]\\
(1-k_1)h\left(\frac{s-s_0}{\alpha}\right) +k_1 &s\in[s_{0}-\alpha,s_0]\\
k_1& s\in[s_0,s_1]\\
\left(k_{i+1}-k_i\right)g\left(\frac{s-s_i}{\alpha}\right) + k_i & s\in[s_i,s_i+\alpha]\\
k_{i+1}& s\in [s_i+\alpha,s_{i+1}]\\
\end{cases}
\]
where $1\leq i\leq m$.

Note that $\tilde{\theta}(s_{m+1})$ could no longer equal $\frac{\pi}{2}$ by the smoothing process. We fix that now. We note that $\tilde{\theta}(s)$ converges uniformly to $\theta(s)$ as $\alpha$ goes to zero. Take $\alpha$ be small enough such that $\tilde{\theta}(s_{m}+\alpha)<\frac{\pi}{2}.$ 

We want  $\tilde{\theta}(s_{m+1}) < \frac{\pi}{2}$. Therefore, if not, then $\tilde{\theta}(s_{m+1}) \geq \frac{\pi}{2}$. Pick a $s^*\in (s_{m}+\alpha,s_{m+1}]$ such that $\frac{\pi}{2}-\alpha<\tilde{\theta}(s^*)<\frac{\pi}{2}$ which exists by the intermediate value theorem and redefine $s_{m+1}=s^*$. 

Let $s_{m+2}=s_{m+1}+2\beta$. On $[s_{m+1},s_{m+2}]$, define
\begin{align*}
    \tilde{k}(s) = \begin{cases}
    -k_{m+1}g\left(\frac{s-s_{m+1}}{\beta}\right)+k_{m+1} & s\in [s_{m+1},s_{m+1}+\beta]\\
    0 & s\in [s_{m+1}+\beta,s_{m+2}],
    \end{cases}
\end{align*}
where $\beta = \frac{\frac{\pi}{2}-\tilde{\theta}(s_{m+1})}{-k_{m+1}(1-H)}$ so that
\[
\int_{s_{m+1}}^{s_{m+1}+\beta} \tilde{k}(s) ds =\frac{\pi}{2}-\tilde{\theta}(s_{m+1}).
\]
Thus, $\tilde{\theta}(s)=\frac{\pi}{2}$ for all $s\in\left[s_{m+1}+\beta,s_{m+2}\right]$. Moreover, we have finished smoothing $\gamma$ to $\tilde{\gamma}$.

Define a half tunnel $A_j=\left\{(y,q)\in X : \left(y,||q||_g\right)\in \tilde{\gamma} \right\}$ with the induced metric. Later, we will glue two half tunnels together to make a tunnel $T_j$. In the following lemma, we record properties of $A_j$ whose proofs are analogous to the ones above.

\begin{lemma} \label{halftunnellemma}
There is a constant $C$ independent of $j$ such that
$(A_j,h_j)$ satisfies the following
\begin{enumerate}
    \item The scalar curvature $R^j$ of $A_j$ satisfies $R^j>\kappa - \frac{1}{j}$.
    \item $\diam{(A_j)} < C(\delta)$.
    \item $\vol{(A_j)} < C(\delta^n)$.
    \item $A_j$ smoothly attaches to $M\setminus B_g(p,2\delta)$ 
    \item The new manifold $(M\setminus B_g(p,2\delta))\sqcup A_j$ is a manifold with boundary.
\end{enumerate}
\end{lemma}
We have constructed half of a tunnel, $A_j$. We now wish to modify the metric at the end of $A_j$ so that it is a product metric of a round sphere and an interval. We follow the same procedure as \cite{D}. Let $a=t(s_{m+1}+\beta)$, $b=t(s_{m+2})$, and $c=r(s_{m+2})$. We note that, by construction, the induced metric on $\{(q,y)\in X: a\leq t\leq b\}$ is $h_0=g_c+dt^2$, where  $g_c$ is the induced metric on $\bS^{n-1}(c)$. Let $h_1=c^2g_{rd}+dt^2$ where $g_{rd}$ is the round metric on the  unit round sphere. Let $\phi(t)=\psi\left(\frac{t-a}{\eta}\right)$ where $\psi(u)$ is a smooth function on $[0,1]$ vanishing near zero, increasing to 1 at $u=\frac{3}{4}$ and equal to 1 for $u>\frac{3}{4}$. Define the metric $h$ for $t\in[a,b]$ as
\[
h(q,y)=g_c(q,y)+\phi(t)\left(c^2g_{rd}-g_c\right)+dt^2.
\]
This metric transitions smoothly between $h_0$ and $h_1$. Note
\[
h-h_0=\phi(t)\left(c^2g_{rd}-g_c\right) = \phi(t)c^2\left(g_{rd}-\frac{1}{c^2}g_c\right)
\]
and that the first and second derivatives of $\phi(t)$ are $O(\eta^{-1})$ and $O(\eta^{-2})$, respectively. So by \Cref{princurv}, we have that the second derivatives of $h-h_0$ are $O(\eta^2)$. Therefore, for $\eta$ small enough, the scalar curvature of $h$ is close to the scalar curvature of $h_0$ which, again by \Cref{princurv}, has scalar curvature larger than $\kappa-\frac{1}{j}$ for small enough $\eta$. Therefore, we have changed the metric at the end of $A_j$ so that it looks like $c^2g_{rd}+dt^2$. Thus, given another ball $B_g(p',2\delta)$ on $M$ we can construct $A'_j$ with a metric at the one end that it looks like $c^2g_{rd}+dt^2$ with the same $c$ by making the same choices in the construction as we did for $A_j$. Now we can immediately glue a cylinder, $([0,d]\times \bS^{n-1}, dt^2+c^2g_{rd})$, connecting $A'_j$ to $A_j$ and so construct the tunnel $T_j$ between $\partial B_g(p',2\delta)$ and  $\partial B_g(p,2\delta)$.

We note that the diameter and volume of the cylinder $([0,d]\times \bS^{n-1}, dt^2+g_{S^{n-1}})$ are bounded by $d$ and $C(n)d\delta^{n-1}$, respectively, where $C(n)$ is a constant that only depends on the dimension. Therefore, we can conclude that $\diam(T_j)$ and $\vol(T_j)$ satisfy the bounds in \Cref{propT}. Therefore, this completes the construction for \Cref{propT}.

\section{Manifolds with shrinking tunnels}\label{theoremsA}
In this section, we will use \Cref{propT} (\nameref{propT}) to construct sequences of manifolds with thinner and thinner long tunnels. Furthermore, we will prove Theorems \ref{minAMV} and \ref{minALL}.

We will need first the following preliminary results.

\begin{prop}\label{mwst1}
There exists a sequence of rotationally symmetric manifolds $M_j=(\bS^n,g_j)$, $n\geq 3$, such that $M_j$ satisfies 
\[
R^j\geq n(n-1)-\frac{1}{j}, \text{ }\diam\left( M_j\right) \leq D, \text{ and } \vol\left( M_j\right) \leq V,
\]
for some constants $0<D, V$ and converges to $M_\infty$ which is the disjoint union of two $n$-spheres.
\end{prop}
\begin{proof}
We will construct the $M_j$ as the connected sum of two standard unit round $n$-spheres for which the tunnel that connects the two spheres gets skinnier as $j$ increases. By \Cref{propT}, we can remove a geodesic ball from both of the spheres and then construct a tunnel $T_j$ connecting the two spheres. Let $(N,h)=(N',h')=(\bS^n,g_{rd})$. Let $j\in\bN$, $j\geq 10$, $d=30$. Define
\[
B_j:=B_{h}\left(p,\frac{2}{j}\right)\subset N,  \text{ and } B':= 
B_{h'}\left(p',\frac{2}{j}\right)\subset N'
\]
where $B_j$ and $B'_j$ are geodesic balls in $N, N'$ respectively. By \Cref{propT}, we can construct a tunnel $T_{j}$ connecting $\partial B_j$ to $\partial B'_j$ and the resulting manifold $M_j$ will have the following properties:
\begin{enumerate}
    \item $M_j=\left(\left( N \sqcup N'\right) \setminus \left(B_j\cup B'_j \right)\right) \sqcup T_j$
    \item $R^j\geq n(n-1)-\frac{1}{j}$.
    \item $M_j\setminus T_j$ is isometric to $(N\setminus B_j) \sqcup (N'\setminus B'_j)$.
    \item $\diam{(M_j)}\leq 4\pi+30$,
    \[
    2\vol_{g_{rd}}{(\bS^n)} - \vol_h{(B_j)} - \vol_{h'}{(B'_j)} \leq \vol{(M_j)}\leq 2\vol_{g_{rd}}{(\bS^n)}+\vol_{g_j}{(T_j)},
    \] 
    and 
    \[
    \lim_{j\to\infty}\vol_h{(B_j)} = \lim_{j\to\infty}\vol_{h'}{(B'_j)} = \lim_{j\to\infty}\vol_{g_j}{(T_j)}=0.
    \]
    In particular, $\lim_{j\to\infty} \vol{(M_j)} = 2\vol_{g_{rd}}{(\bS^n)}. $
\end{enumerate}
By \Cref{thm4.6}, we have the intrinsic flat distance between $M_j$ and $N\sqcup N'$ is 
\[
d_\mathcal{F}(M_j,N_1\sqcup N_2)\lesssim \frac{1}{j}\left(\vol_{g{rd}}(\bS^n)+\vol_{g_{rd}}(\bS^{n-1})\right)+ \vol_h{(B_j)} + \vol_{h'}{(B'_j)}+ \vol_{g_j}{(T_j)}.
\]
As $j\to\infty0$, we that $\vol_h{(B_j)}$, $\vol_{h'}{(B'_j)}$, and $\vol_{g_j}{(T_j)}$ go to zero. Therefore, we conclude that $M_j$ converges to $N\sqcup N'$ in the $\mathcal{VF}$ sense.
\end{proof}

\begin{remark}
From the construction in \Cref{propT} (\nameref{propT}) we see that $M^n_j=([0,D_j]\times \bS^{n-1},g_j)$ defined above is rotationally symmetric. Moreover, near $\{0\}\times\bS^{n-1}$ and $\{D_j\}\times\bS^{n-1}$, we have that $M_j^n$ is isometric to the standard unit round $n$-sphere. In particular, the metric takes the form $g_j=dt^2+\sin^2(\rho_j(t))g_{\bS^{n-1}}$ where $D_j$ is the diameter of $M_j$ and for $\rho_j:[0,D_j]\to [0,\infty)$ is a smooth function with the following properties. Recall $\tilde{\gamma}_j=(\tilde{t}_j(s),\tilde{r}_j(s))$ to be the curve define in \Cref{halftunnellemma} that defines the half tunnel $A_j$. Then 
\[
\rho(t) =\begin{cases}
\hat{r}(t), & t\in \left[0,\frac{1}{2}D_j\right]\\
\hat{r}(D- t), & t\in\left[\frac{1}{2}D_j,D_j\right]
\end{cases}
\text{ and } \hspace{2.5pt}
\hat{r}(t) =\begin{cases}
\vspace{2pt}\pi-t, & t\in \left[0,\pi-\frac{2}{j}\right]\\
\tilde{r}(t+(\delta-\pi)), & t\in\left[\pi-\frac{2}{j},\frac{1}{2}D\right].
\end{cases}
\]
\end{remark}

We will now construct smooth $1$-Lipschitz maps $F_j:M^n_j\to (\bS^n,g_{rd})$. But first, we need the following result based on the mollification in \cite[Section 3]{M}. Since our lemma varies slightly from what is stated in \cite{M} we provide an analogous proof.

\begin{lemma}\label{convolution}
Let $h:\bR\to\bR$ be an L-Lipschitz continuous function such that
\[
h(t)=\begin{cases}
h_+(t), & t\in (0,\infty)\\
h_-(t), & t\in (-\infty,0),
\end{cases}
\]
where $h_+$ and $h_-$ are smooth functions.
Then for small enough $\eps>0$ there exists a function $h_\eps:\bR\to\bR$ such that 
\[
||h_\eps(t)-h(t)||_{C^2} \lesssim \eps^2, \hspace{.5em}  h'_\eps(t) \leq \sup\{h'(t):t\in\bR\setminus \{0\}\}, \hspace{.5em} \text{and} \hspace{.5em}  |h'_\eps(t)|\leq L.
\]
\end{lemma}

\begin{proof} Let $0<\eps_0<1$. We will restrict our attention to $(-\eps_0,\eps_0)$. Let $\varphi\in C^\infty_c([-1,1])$ be the standard mollifier in $\bR$ such that 
\[
0\leq\varphi\leq 1 \qquad \text{and} \qquad \int_{-1}^1 \varphi(t)dt=1.
\]
Let $\sigma(t)\in C^\infty_c\left(\left[-\frac{1}{2},\frac{1}{2}\right]\right)$ be another bump function such that
\begin{align*}
&0\leq \sigma(t) \leq \frac{1}{100}\text{ for }  t\in\bR,\\
&\sigma(t)=\frac{1}{100} \text{ for } |t|<\frac{1}{4},\\
&0<\sigma(t) \leq \frac{1}{100} \text{ for } \frac{1}{4}<|t|<\frac{1}{2}.    
\end{align*}

Let $0<\eps<\frac{1}{10}\eps_0$. Define $\sigma_\eps(t) = \eps^3\sigma\left(\frac{t}{\eps}\right).$ Moreover, define
\begin{equation}\label{conv}
\begin{split}
    h_\delta(t) &= \int_\bR h(t-\sigma_\delta(t)s)\varphi(s)ds,\qquad t\in\left(-\eps_0, \eps_0\right)\\
  &=\begin{cases}
  \int_\bR h(s) \cdot \frac{1}{\sigma_\delta(t)} \varphi\left(\frac{t-s}{\sigma_\delta(t)}\right) ds,  & \sigma_\delta(t) >0\\
  h(t), & \sigma_\delta(t)=0.
  \end{cases}
  \end{split}
\end{equation}
Now we want to compute $h'_\delta(s)$. For $\left|t\right|>\frac{\eps^3}{100}$,
\begin{align*}
         h'_{\eps}(t) &= \frac{d}{dt} \int_{\bR} h(t-\sigma_{\eps}(t)s)\varphi(s)ds \\
        &=\int_{\bR} h'(t-\sigma_{\eps}(t)s)\left(1-s\eps^2\sigma'\left(\frac{t}{\eps}\right)\right)\varphi(s)ds.
\end{align*}

For $|t|<\frac{\eps}{4}$,
\begin{align*}
        h'_\delta(t) &= \frac{d}{dt}  \int_\bR h(s) \cdot \frac{1}{\sigma_\eps(t)} \varphi\left(\frac{t-s}{\sigma_\eps(t)}\right) ds\\
        &=\int_\bR h(s)\cdot \frac{d}{dt}\left(\frac{1}{\sigma_\eps(t)} \varphi\left(\frac{t-s}{\sigma_\eps(t)}\right)\right)ds \\
        &= \int_\bR h(s)\cdot \frac{d}{dt}\left(\frac{100}{\eps^3} \varphi\left(\frac{100(t-s)}{\eps^3}\right)\right)ds\\
        &= (-1)\cdot \int_\bR h(s)\cdot \frac{d}{ds}\left(\frac{100}{\eps^3} \varphi\left(\frac{100(t-s)}{\eps^3}\right)\right)ds\\
        &= (-1)\cdot\int_{-\infty}^0 h_-(s)\cdot \frac{d}{ds}\left(\frac{100}{\eps^3} \varphi\left(\frac{100(t-s)}{\eps^3}\right)\right)ds \\
        &\qquad+ (-1)\cdot\int_0^\infty h_+(s)\cdot \frac{d}{ds}\left(\frac{100}{\eps^3} \varphi\left(\frac{100(t-s)}{\eps^3}\right)\right)ds\\
        &=\int_{-\infty}^0 h'_-(s)\cdot \left(\frac{100}{\eps^3} \varphi\left(\frac{100(t-s)}{\eps^3}\right)\right)ds \\
        &\qquad+ \int_0^\infty h'_+(s)\cdot \left(\frac{100}{\eps^3} \varphi\left(\frac{100(t-s)}{\eps^3}\right)\right)ds \\
        &=\int_\bR h'(s) \cdot \left(\frac{100}{\eps^3} \varphi\left(\frac{100(t-s)}{\eps^3}\right)\right)ds\\
        &=\int_\bR h'(t-\sigma_\eps(t)s)\varphi(s)ds.
\end{align*}

Now note for $|t|<\frac{\eps}{4}$ that $\sigma_\eps$ is a constant function; therefore, for all $t\in(-\eps_0,\eps_0)$
\begin{equation}\label{der}
  h'_\eps(t)=\int_\bR h'(t-\sigma_\eps(t)s)\left(1-s\eps^2\sigma'\left(\frac{t}{\eps}\right)\right)\varphi(s)ds .
\end{equation}
By (\ref{conv}) and (\ref{der}) we have
\begin{align*}
         ||h_\eps(t)-h(t)||_{u}&\leq \int_\bR ||h(t-\sigma_\eps(t)s) -h(t)||_{u} \varphi(s)ds\\
         &\lesssim \eps^3.
\end{align*}
and
\begin{align*}
    ||h'_\eps(t)-h'(t)||_u &\leq \int_\bR ||h'(t-\sigma_\eps(t)s)-h'(t)||_u \varphi(s)ds \\
    &\qquad+ \int_\bR \left|\left|h'(t-\sigma_\eps(t)s)s\eps^2\sigma'\left(\frac{t}{\eps}\right)\right|\right|_u \varphi(s)ds\\
    &\lesssim \eps^3 + \eps^2\int_\bR \left|\left|h'(t-\sigma_\eps(t)s)\sigma'\left(\frac{t}{\eps}\right)\right|\right|_u \varphi(s)ds\\
   &\lesssim \eps^2.
\end{align*}
Lastly, note that
\begin{align*}
    |h'_\eps(t)|&=\int_\bR |h'(t-\sigma_\eps(t)s)|\left|\left(1-s\eps^2\sigma'\left(\frac{t}{\eps}\right)\right)\right||\varphi(s)|ds\\
    &\leq L\int_\bR \left(1-s\eps^2\sigma'\left(\frac{t}{\eps}\right)\right)(\varphi(s))ds\\
    &=L\left(1-\eps^2\sigma'\left(\frac{t}{\eps}\right) \int_\bR s\varphi(s)ds\right)\\
    &\leq L,
\end{align*}
where the first inequality follows if $\eps$ is small enough and the last inequality follows since $s\varphi(s)$ is an odd function. Moreover, redoing this computation without the absolute values shows that $h'_\eps(t) \leq \sup\{h'(t):t\in\bR\setminus\{0\}\}$.
\end{proof}

Now we are ready to construct smooth $1$-Lipschitz maps $F_j:M^n_j\to (\bS^n,g_{rd})$.

\begin{lemma}\label{mwst2}
There exists a function $F_j:M^n_j\to \bS^n$ that is a $1$-Lipschitz diffeomorphism with $\deg F_j\neq0$
\end{lemma}
\begin{proof}
First define a decreasing 1-Lipschitz function $f_j:[0,D_j]\to [0,\pi]$. 
\[
f_j(t)=\begin{cases}
\pi-t, & t\in [0,t_j]\\
a_j(t-t_j)+b_j, & t\in[t_j,D_j],
\end{cases}
\]
where $a_j=\frac{-\pi+t_j}{D_j-t_j}$, $b_j=\pi-t_j$, and $t_j$ is chosen so that $f_j(t_j)=\frac{1}{10}\rho\left(\frac{1}{2}D_j\right)$. Note $\rho\left(\frac{1}{2}D_j\right)$ is the radius of the cylindrical part of the tunnel which is also the minimum that $\rho_j(t)$ attains on $\left[\frac{\pi}{2},D_j-\frac{\pi}{2}\right]$.

By \Cref{convolution}, we can smooth $f_j$ to $f_{j,\eps}$ by choosing $\eps_0$ and $\eps$ small enough. And so define $F_{j,\eps}(t,\theta)=(f_{j,\eps}(t),\theta)$. Since $f'_{j,\eps}(t)<0$ and $f_{j,\eps}$ is a bijection, we have that $F_{j,\eps}$ is a diffeomorphism.
We want to show that for all $v\in TM_j$
\[
F_{j,\eps}^*g_{rd}(v,v) \leq g_j(v,v).
\]

Note that
\[
F_{j,\eps}^*g_{rd}=\left(f_{j,\eps}'(t)\right)^2dt^2+\sin^2(f_{j,\eps}(t))g_{\bS^{n-1}}.
\]
and 
\[
g_j=dt^2+\sin^2(\rho_j(t))g_{\bS^{n-1}}.
\]
First by (\ref{curve}) and \Cref{convolution} we know that $|f'_{j,\eps}(t)|\leq 1$ for all $t$. Now we will show that $\sin^2(f_{j,\eps}(t))\leq \sin^2(\rho_j(t)).$

On $\left[0,\pi-t_{j}-20\eps\right]$ we have by (\ref{curve}) that 
\[
\rho_j(t)=\pi-\int_0^t\cos\left(\theta_j (u)\right) du \geq \pi -t =f_j(t) = f_{j,\eps}(t).
\]
On $\left[\pi-t_{j}-20\eps,\pi-t_j\right]$ we have 
\[
\rho_j(t)=\pi-\int_0^t\cos\left(\theta_j (u)\right) du > \pi -t =f_j(t)
\]
and so for small enough $\eps$, we have that $f_{j,\eps}(t)$ will also satisfy this inequality.

On $\left[\pi-t_{j}, D_j-\frac{\pi}{2}\right]$, we have that $f_j(t) \leq  \frac{1}{10}\rho\left(\frac{1}{2}D_j\right)$ and that $\frac{1}{10}\rho\left(\frac{1}{2}D_j\right)< \rho_j(t) \leq \frac{\pi}{2}$. Therefore, $\sin^2(f_j(t))\leq \sin^2(\rho_j(t))$ on $\left[0, D_j-\frac{\pi}{2}\right]$.

Lastly on $\left[ D_j-\frac{\pi}{2},D_j\right]$ we have the following: $\rho_j(t)=\pi-D_j+t$ and $f_{j,\eps}(t)=f_j(t)=a_j(t-t_j)+b_j$ by the construction. Moreover,
\[
-f_j(t)+\pi \geq \rho_j(t)
\]
since if we define $\psi_j(t)=\rho_j(t)+f_j(t)-\pi$, then we see that $\psi'(t)\geq 0$ and $\psi(D_j)=0$. We also note on $\left[ D_j-\frac{\pi}{2},D_j\right]$ that $\frac{\pi}{2}\leq -f_j(t)+\pi \leq \pi$ and $\frac{\pi}{2}\leq \rho_j(t) \leq \pi$. Therefore, we conclude that $\sin^2(f_{j,\eps}(t))=\sin^2(-f_{j,\eps}(t)+\pi)\leq \sin^2(\rho_j(t))$ on $\left[ D_j-\frac{\pi}{2},D_j\right]$.

Thus, for all $v\in TM_j$ we have 
\[
F_{j,\eps}^*g_{rd}(v,v) \leq g_j(v,v),
\]
which implies 
\[
 \ell_{\bS^{n}}\left(F_{j,\eps}\circ c\right) \leq \ell_{M_j}(c)
\]
where $c:[0,1]\to (\bS^n,g_{rd})$ is a path connecting $p$ and $q$. This implies that
\[
d_{\bS^{n}}\left(F_{j,\eps}(p),F_{j,\eps}(q)\right) \leq  d_{M_j}(p,q).
\]
Thus, we have that $F_{j,\eps}$ is 1-Lipschitz. Moreover, $\deg F_{j,\eps} \neq 0$ since  $F_{j,\eps}$ is a diffeomorphism.
\end{proof}

\begin{lemma}\label{widlem}
Let $(\bS^3,g_1), (\bS^3,g_2)$ be 3-spheres such that there exists a diffeomorphism $F:(\bS^3,g_1)\to (\bS^3,g_2)$ that is 1-Lipschitz and is isotopic to the identity then 
\[
\mathrm{width}(\bS^3,g_2)\leq \mathrm{width}(\bS^3,g_1).
\]
\end{lemma}
\begin{proof}
By the definition of $\mathrm{width}$ for any $\delta>0$ there exists $\{\Sigma_t\}$ such that \[
\sup_t |\Sigma_t|_1 < \mathrm{width}(\bS^3,g_1)+\delta;
\]
therefore,
\[
\mathrm{width}(\bS^3,g_2)\leq \sup_t |F(\Sigma_t)|_2\leq \sup_t |\Sigma_t|_1 \leq \mathrm{width}(\bS^3,g_1) +\delta
\]
where  the first inequality follows since $F(\Sigma_t)\in \Lambda'$ and the second inequality follows since $F$ is 1-Lipschitz.
\end{proof}

\begin{proof}[Proof of \Cref{minALL}]
Let $M_j$ be as in \Cref{mwst1}; therefore, $M^n_j\to M_\infty$ where $M_\infty$ is the disjoint union of two spheres. Let $F_j:M^n_j\to \bS^n$ be as in \Cref{mwst2}. Then by Arzela-Ascoli \Cref{thm6.1} there is a subsequence $F_{j_k}$ that converges to a $1$-Lipschitz map 
\[
F_\infty: M_\infty \to \bS^n.
\]
This map is not a Riemannian isometry since $\bS^n$ is connected and $N\sqcup N'$ is not.
\end{proof}

\begin{proof}[Proof of \Cref{minAMV}]
Let $M^3_j$ be as in \Cref{mwst1}; therefore, $M^3_j\to M_\infty$ in $\mathcal{VF}$-sense where $M_\infty$ is the disjoint union of two spheres. Let $F_j:M^3_j\to \bS^n$ be as in \Cref{mwst2} and define $\tilde{F}_j(r,\theta)=F_j(D_j-r,\theta)$. Consider the diffeomorphism 
\[
\Phi:[0,D_j]\times \bS^{2}\to [0,\pi]\times \bS^{2}, \qquad \Phi(r,\theta)= \left(\frac{\pi}{D_j}r,\theta\right)
\]
Note that $\Phi$ is an isometry between $([0,D_j]\times \bS^{n-1}, \Phi^*(dr^2 +\sin^2(r)g_{\bS^{n-1}}))$ and \\$([0,\pi]\times \bS^{n-1}, dr^2 +\sin^2(r)g_{\bS^{n-1}})$. And now consider 
\[
(\Phi^{-1}\circ \tilde{F}_j) (r,\theta)=\left(\frac{D_j}{\pi}f_j(D_j-r),\theta\right).
\]
This map is a 1-Lipschitz orientation preserving diffeomorphism from $M^n_j$ to the round $n$-sphere and
 $\Phi^{-1}\circ F_j$ is isotopic to the identity. Therefore, by \Cref{widlem} we have that $\mathrm{width}(M^3_j)\geq 4\pi$. 
\end{proof}

\section{Manifolds with many wells}\label{theoremsB}
In this section, we will use \Cref{propW} (\nameref{propW}) to construct sequences of manifolds with many wells. Furthermore, we will prove Theorems \ref{spikesMV} and \ref{spikesLL}. 

\begin{theorem}\label{mwmw1}
Let $(M^n,g)$ be a closed Riemannian manifold of dimension $n\geq 3$ with scalar curvature $R\geq\kappa$. Then there exists a sequence of Riemannian manifolds $M^n_j=(M^n,g_j)$ such that $R^j\geq \kappa-\frac{1}{j}$ and $M^n_j$ converge in the ${\mathrm{VADB}}$-sense and $\mathcal{VF}$-sense to $M^n$ but has no convergent subsequence in the $\mathrm{GH}$-topology.
\end{theorem}
\begin{proof}
Define
\[
X_j=\left\{(B(p_i^j,\delta_j),g)\right\}_{i=1}^j
\]
to be a collection of disjoint geodesic balls in $M^n$ where $0<\delta_j<\frac{1}{j}$ is chosen small enough so that by \Cref{propW} we replace each $B(p_i^j,\delta_j)$ with the well $W_{i,j}=(B(p_i^j,\delta_j),g_j)$ such that the scalar curvature of each of the wells satisfies $R^j>\kappa-\frac{1}{j}$. Moreover, choose $d=\frac{1}{2}$ in \Cref{propW} so that $\diam({W_{i,j}})\geq \frac{1}{2}$. Call the resulting manifold $M^n_j=(M^n,g_j)$. Now we note that 
\begin{align*}
    \lim_{j\to\infty} \vol_j{(M^n_j)} &= \lim_{j\to \infty} \vol_g{(M^n)} - \sum_{i=1}^j\vol_g{(B(p_i^j,\delta_j))} + \sum_{i=1}^j\vol_j{(W_{i,j})}.
\end{align*}
   Thus, by \Cref{propW} and \Cref{propvol}
     \[
     \lim_{j\to \infty} \vol_g{(M^n)} - jC\delta_j^n\leq \lim_{j\to\infty} \vol_j{(M^n_j)}\leq  \lim_{j\to \infty} \vol_g{(M^n)} + Cj\left(\delta_j^n+\frac{\delta_j^{n-1}}{2}\right).
     \]
and so 
\[
 \lim_{j\to\infty} \vol_j{(M^n_j)} =\vol_g{(M^n)}.
\]
Also by \Cref{propW} and the triangle inequality, we have that 
\[
\diam\left(M^n_j\right) \leq \diam \left((M^n,g)\right) + 2\diam{(W_j)} \leq \diam \left((M^n,g)\right) + 2\left(C+\frac{1}{2}\right).
\]
so the diameters are uniformly bounded.

Consider the identity map $id:(M^n,g_j)\to (M^n,g)$. Denote $id^*g_j=g_j$. Now by construction and \Cref{smoothattach} we have for any $p\in W_{i,j}$ that 
\[
g(v,v) \leq  g_j(v,v) \text{ for all } v\in T_pM 
\]
because $g_j=dF^2_j+g$ and if $p\notin W_{i,j}$ then $g(v,v) =  g_j(v,v) \text{ for all } v\in T_pM$.

Therefore, $M^n_j$ converges to $(M^n,g)$ in the $\mathrm{VADB}$-sense and by \Cref{VADB} we have that $M^n_j$ converges to $(M^n,g)$ in the $\mathcal{VF}$-sense. 

Fix $\eps_0<\frac{1}{4}$. Note that $\eps_0<d$ and so  $B(p^j_i,\eps_0)\subset M^n_j$ are disjoint. Therefore, 
\[
j< \text{Cov}_j(\eps_0)
\]
and so as $j\to \infty$ we have $\text{Cov}_j(\eps_0) \to \infty$ so by \Cref{GH} that $M_j$ does not converge in the $\mathrm{GH}$-sense.

\end{proof}

\begin{proof}[Proof of \Cref{spikesLL}]
Consider the round $n$-sphere $(\bS^n,g_{rd})$. By \Cref{mwmw1} we see that there exists a sequence $M_j=(\bS^n,g_j)$ with scalar curvature $R^j\geq n(n-1)-\frac{1}{j}$ such that $M_j\to (\bS^3,g_{rd})$ in the $VADB$ and $\mathcal{VF}$-sense but has no convergent subsequence in the $\mathrm{GH}$-topology. Furthermore, the identity map $id:(\bS^n,g_j)\to (\bS^n,g_{rd})$ is smooth 1-Lipschitz diffeomorphism.
\end{proof}

\begin{proof}[Proof of \Cref{spikesMV}]
Consider the round 3-sphere $(\bS^3,g_{rd})$. By \Cref{mwmw1} we see that there exists a sequence $(\bS^3,g_j)$ with scalar curvature $R^j\geq 6-\frac{1}{j}$ such that $(\bS^3,g_j)\to (\bS^3,g_{rd})$ in the $VADB$ and $\mathcal{VF}$-sense but has no convergent subsequence in the $\mathrm{GH}$-topology. Moreover, the identity map $id:(\bS^3,g_j)\to (\bS^3,g_{rd})$ is 1-Lipschitz and by \Cref{widlem} we have that $\mathrm{width}(\bS^3,g_j)\geq 4\pi$.
\end{proof}

\begin{proof}[Proof of \Cref{noncompact}]
Let $\kappa>0$ and let $(M^n,g)$ be the round sphere of curvature $\frac{2\kappa}{n(n-1)}$. Let $\{p_j\}_{j=1}^\infty\subset M^n$ be a sequence of points on a geodesic converging to a point $p_\infty$. Define
\[
\left\{B(p_j,\delta_j)\right\}_{j=1}^\infty
\]
to be a collection of disjoint geodesic balls in $M^n$ where $0<\delta_j<\frac{1}{2^j}$ is chosen small enough so that by \Cref{propW} there exists a well $W_{j}=(B(p_j,\delta_j),g_j)$ such that the scalar curvature of each of the wells satisfies $R^j>2\kappa\left(1-\frac{1}{10j}\right)>\kappa$. Let $\{d_j\}_{j=1}^\infty \subset [2,10]$ be a strictly increasing sequence of positive numbers, and choose $d=d_j$ in \Cref{propW} so that $\diam({W_{j}})\geq d_j$. Now define $M^n_i$ to be the Riemannian obtained by replacing the first $i$ balls with the corresponding first $i$ wells, i.e.,
\[
M^n_i= \left(M^n\setminus \bigcup_{j=1}^i B(p_j,\delta_j) \right) \sqcup \bigcup_{j=1}^i W_j.
\]
We note that $M^n_i$ has scalar curvature strictly larger than $\kappa$. We also have by \Cref{propW} that
\[
\diam(M^n_i) \leq 25C
\]
and
\begin{align*}
\vol(M^n_i) &\leq \vol(M^n) + \sum_{j=1}^\infty \vol (W_j)\\ 
&\leq \vol(M^n) + C\left(\sum_{j=1}^\infty \frac{1}{2^{nj}} + 10\sum_{j=1}^\infty \frac{1}{2^{(n-1)j}}\right)\\
&\leq \vol(M^n) + 11C.
\end{align*}
Now we will define $M_\infty$ to be
\[
M_\infty = \left(M^n\setminus \bigcup_{j=1}^\infty B(p_j,\delta_j) \right) \sqcup \bigcup_{j=1}^\infty W_j
\]
with its induced length metric and natural current structure $T_\infty$. Therefore, we have that $\vol(M^n_i) \to \vol(M_\infty).$ Let $E_j\subset W_j$ be a ball centered at $p_j$ of radius 1 and so $M_\infty$ is noncompact since it contains infinitely many disjoint balls of radius 1.

We will show that $M^n_i$ converges to $M_\infty$ in an analogous many to  \cite[Example A.11]{SW}. Let $\eps_i=d_{M^n}(p_i,p_\infty)$ and note that if $\tilde{B}_i =B(p_\infty, \eps_i-\delta_i)$, then there is an isometry, $\varphi:V_i \to V'_i$ where $U_i=M_i^n\setminus \tilde{B}_i\subset M_i$ and $U'_i\subset M_\infty$.  By \cite[Lemma A.2]{SW}, there exists a metric space $Z$ such that 
\begin{align*}
d_F^Z(M^n_i,M_\infty) &\leq \vol(M^n_i\setminus U_i) +\vol (M_\infty \setminus U'_i) \\
&\quad\quad + \vol({U}_i)\left(\sqrt{2\diam_{M^n_i}(\partial U_i) \diam_{M^n_i}( U_i)} + \diam_{M^n_i}(\partial U_i) \right) \\
&\quad\quad + \vol(U'_i)\left(\sqrt{2\diam_{M^n_i}(\partial U'_i) \diam_{M^n_i}( U'_i)} + \diam_{M^n_i}(\partial U'_i) \right).
\end{align*}
We note that
\[
\vol(M^n_i\setminus U_i)\leq \pi(\eps_i-\delta_i)^n, \quad \quad \vol (M_\infty \setminus U'_i) \leq C\left(\sum_{j=i}^\infty \frac{1}{2^{nj}} + 10\sum_{j=i}^\infty \frac{1}{2^{(n-1)j}}\right).
\]
Also, $\diam(\partial U_i)$ and $\diam(\partial U'_i)$ converge to zero. Therefore, the right-hand side of the inequality above goes to zero as $i\to\infty$. We conclude then that $M^n_i$ converges to $M_\infty$ in the $\mathcal{VF}$-sense.
\end{proof}

\section{Sewing Manifolds}\label{sewing}
We are able to generalize the sewing examples of Basilio, Dodziuk, and Sormani found in \cite{BDS} and \cite{BS}. There are two methods of sewing developed in \cite{BS}. Method I generalizes the curve sewing construction of \cite{BDS}. Here we will extend the construction using \Cref{propT} (\nameref{propT}). We start with Method I which says that given a fixed manifold one can tightly sew a compact region to a point.
\begin{prop}\label{sewprop1}
Let $(M^n,g)$ be a complete Riemannian manifold, and $A_0\subset M$ a compact subset with an even number of points $p_i\in A_0$, $i=1,\ldots,n$ with pairwise disjoint balls $B(p_i,2\delta)$ with scalar curvature greater than $\kappa$. For small enough $\delta>0$, define $A_\delta:=T_\delta(A_0)$ and 
\[
A'_\delta =A_\delta\setminus \left(\bigcup_{i=1}^nB(p_i,\delta)\right) \sqcup \bigcup_{i=1}^\frac{n}{2} T_i
\]
where $T_i$ are tunnels as in \Cref{propT} (\nameref{propT}) connecting $\partial B(p_{2j+1},\delta)$ and $\partial B(p_{2j+2},\delta)$ for $j=0,1,\ldots,\frac{n}{2}-1$. Then given any $\eps$, shrinking $\delta$ further, if necessary, we may create a new complete Riemannian manifold, $(N^n,h)$, 
\[
N^n=(M^n\setminus A_\delta)\sqcup A'_\delta
\]
satisfying
\[
\vol{(A_\delta)} -\eps \leq \vol{(A'_\delta)} \leq \vol{(A_\delta)} +\eps
\]
and 
\[
\vol{(M)} -\eps \leq \vol{(N)} \leq \vol{(M)} +\eps
\]
If, in addition, $M$ has scalar curvature, $R^M\geq \kappa$, then $N$ has scalar curvature, $R^N\geq \kappa-\eps$. If $\partial M\neq \emptyset$, the balls avoid the boundary and $\partial M$ is isometric to $\partial N$.
\end{prop}
\begin{proof}
The proof follows from the proof of \cite[ Proposition 3.1]{BDS} while using \Cref{propT} (\nameref{propT}) and \Cref{propvol}.
\end{proof}
\begin{prop}\label{sewprop2}
Let $(M^n,g)$ be a complete Riemannian manifold and $A_0\subset M$. Let $A_a=T_a(A_0)$ be a tubular neighborhood of $A_0$. Assume that there is an $a>0$ such that $A_a$ has scalar curvature greater than $\kappa$. Let $r\in(0,a)$. Given $\eps>0$, there exists $\delta=\delta(A_0,\kappa,r,\eps)\in(0,r)$ and there exists even $n=\bar{n}(\bar{n}-1)$ depending on $A_0, \kappa, \eps$, and $r$ and points $p_1,\ldots,p_n \in A_0$  with $B(p_i,\delta)$ pairwise disjoint such that we can ``sew the region tightly" to create a new complete Riemannian manifold $(N^n,h)$,
\[
N=(M\setminus A_r)\sqcup A'_r,
\]
as in \Cref{sewprop1}, with
\[
A'_\delta = A_\delta \setminus \left(\bigcup_{i=1}^{2n}B(p_i,\delta)\right) \sqcup \bigcup_{j=0}^{n-1} T_{2j+1}.
\]
Moreover,
\[
\vol{(A'_r)} \leq \vol{(A_r)} +\eps
\]
and 
\[
\vol{(N)} \leq \vol{(M)} +\eps
\]
and there is a constant $c>0$ such that
\[
\diam{(A'_r)} \leq cr.
\]
we say we have sewn the region $A_0$ arbitrarily tight. If $M$ has scalar curvature $R^M \geq \kappa$, then $N$ has scalar curvature $R^N\geq \kappa-\eps$. If $\partial M\neq \emptyset$, the balls avoid the boundary, and $\partial M$ is isometric to $\partial N$.
\end{prop}
\begin{proof}
The proof follows from the proof of \cite[Proposition 3.6]{BS} while using Propositions \ref{propT}, \ref{sewprop1}, and \Cref{princurv}.
\end{proof}
These statements allow us to construct sequences of manifolds with scalar curvature greater than $\kappa$ which converge to a pulled metric space in a similar manner as in \cite{BS}. We recall the following definition from \cite{BS}.
\begin{deff}\label{sewdef}
Let $(M^n,g)$ be a Riemannian manifold with a compact set $A_0\subset M$ with tubular neighborhood $A_a=T_a(A_0)$ satisfying the hypotheses of \Cref{sewprop2}. We can construct its \emph{sequence of increasingly tightly sewn manifolds}, $(N^n_j,g_j)$, by applying \Cref{sewprop2} taking $\eps=\eps_j\to0$, $n=n_j\to\infty$, and $\delta=\delta_j\to 0$ to create each sewn manifold $N^n=N^n_j$ and the edited regions $A'_\delta=A'_{\delta_j}$ which we simply denote $A'_j$. Since these sequences $N_j$ are created using \Cref{sewprop2}, they have scalar curvature greater than $\kappa-\eps_j$ when $M$ has scalar curvature greater than $\kappa$ and $\partial N_j=\partial M$ whenever $\partial M\neq \emptyset$. 
\end{deff}
\begin{theorem}\label{sewthm}
The sequence $N_j$, as in \Cref{sewdef} assuming $M^n$ is compact and $A_0$ is a compact embedded submanifold of dimension 1 to $n$, converges in the Gromov-Hausdorff sense and the intrinsic flat sense to $N_\infty$, which is a metric space created by pulling the region $A_0$ to a point. If, in addition, $\mathcal{H}^{n-1}(A_0)=0$ then $N_j$ also converges in the metric measure sense to $N_\infty.$
\end{theorem}
\begin{proof}
The proof follows from the proof of \cite[Theorem 3.8]{BS} while using \Cref{sewprop2}.
\end{proof}
Now we can prove \Cref{sewingthm}.
\begin{proof}[Proof of \Cref{sewingthm}]
Let $S$ be a simply connected space form of dimension $n$ and constant curvature $\frac{\kappa}{n(n-1)}$ and $\Sigma^m$ be a constant curvature $m$-dimensional sphere, $1\leq m \leq n-1$. We note that there exists an embedding of $\Sigma^m$ into $S$. Let $(N^n_j,g_j)$ be a sequence of manifolds constructed from $S$ sewn along an embedded $\Sigma^m$ with $\delta=\delta_j\to 0$ as in \Cref{sewprop2} and the scalar curvature $R^j\geq \kappa -\frac{1}{j}$. Then by \Cref{sewthm} we have 
\[
N_j\xrightarrow{mGH} N_\infty \text{ and } N_j  \xrightarrow{\mathcal{F}} N_\infty 
\]
where $N_\infty$ is the metric space created by taking $S$ and pulling a $\Sigma^m$ to a point. Moreover, at the pulled point $p_0\in N_\infty$ we have 
\[
wR(p_0)=\lim_{r\to 0} 6(n+2) \frac{\vol_{\bE^n}{B(0,r)}-\mathcal{H}^n(B(p_0,r))}{r^2\cdot\vol_{\bE^n}{B(0,r)} }=-\infty.
\]
We can see this because
\[
\vol_{N_\infty}{(B(p_0,r))} = \mathcal{H}^n_{N_\infty}(B(p_0,r)) = \mathcal{H}^n_{N_\infty}(B(p_0,r)\setminus \{p_0\}) =\mathcal{H}^n_{\bS_\kappa^n}(T_r(\bS^m)).
\]
Moreover, there is a constant $C(n,m,\kappa)$ such that
\[
\lim_{r\to0 } \frac{\mathcal{H}^n_{\bS_\kappa^n}(T_r(\Sigma^m))}{Cr^{n-m}} =1.
\]
We conclude that
\[
wR(p_0)= \lim_{r\to 0} 6(n+2) \frac{\omega_nr^n-Cr^{n-m}}{\omega_nr^{n+2} } =-\infty.
\]
\end{proof} 
Moreover, using \Cref{propT} we are to extend Method II for sewing manifolds in \cite{BS} to the setting where scalar curvature is bounded below. In Method II, given a sequence of Riemannian  manifolds whose limit is a Riemannian, then one can create a new sequence where the sewing occurs along the sequence.
\begin{theorem}
Let $M^n_j$ be a sequence of compact Riemannian manifolds each with a compact region $A_{j,0}\subset M^3_j$ with tubular neighborhood, $A_j$, with scalar curvature greater than $\kappa$ satisfying the hypotheses of \Cref{sewprop2}. We assume $M^n_j$ converge in the biLipschitz sense to $M^n_\infty$ and the regions $A_{j,0}$ converge to a compact set $A_{\infty,0}\subset M_\infty^n$ in the sense that there exists biLipschitz maps
\[
\psi_j:M^n_j\to M^n_\infty
\]
such that
\[
L_j=\log\left(\emph{Lip}(\psi_j)\right)+\log\left(\emph{Lip}\left(\psi^{-1}_j\right)\right)\to 0
\]
and $\psi_j(A_{j,0})=A_{\infty,0}$.
Then there exists $\delta_j\to 0$ and applying \Cref{sewprop2} to $M^n=M^n_j$ to sew the regions $A_0=A_{j,0}$ with $\delta=\delta_j$, to obtain sewn manifolds $N^n=N^n_j$, we obtain a sequence $N^n_j$ such that 
\[
N_j^n\xrightarrow{GH} N_\infty \text{ and } N_j^n\xrightarrow{\mathcal{F}} N_{\infty,0},
\]
where $\bar{N}_{\infty,0}=N_\infty$ and $N_\infty$ is the metric space created by taking $M^n_\infty$ and pulling the region $A_{\infty,0}$ to a point.

Moreover, if the regions $A_{j,0}$ satisfy $\mathcal{H}^n(A_{j,0})=0$, the the sequence $N_j^n$ also converges in the metric measure sense
\[
N_j^n\xrightarrow{mGH} N_\infty.
\]
\end{theorem}
\begin{proof}
The proof follows from the proof of \cite[Theorem 5.1]{BS} while using \Cref{sewprop2} and \Cref{sewthm}
\end{proof}

\section{Intrinsic Flat limit with no geodesics}\label{nogeo}
We are able to generalize the result of Basilio, Kazaras, and Sormani from \cite{BKS} which shows the intrinsic flat limit of Riemannian manifolds need not be geodesically complete. This follows from \Cref{propT} (\nameref{propT}) and the pipe-filling technique \cite[Theorem 3.1]{BKS}. In particular:
\begin{theorem}
There is a sequence of closed, oriented, Riemannian manifolds $(M^n_j,g_j)$, $n\geq 3$, such that the corresponding integral current spaces converge in the intrinsic flat sense to 
\[
M_\infty = \left(N,d_{\bE^{n+1}}, \int_{N}\right),
\]
where $N$ is the round $n$-sphere of curvature $\frac{2\kappa}{n(n-1)}$ and $d_{\bE^{n+1}}$ is the Euclidean distance induced from the standard embedding of $N$  into $\bE^{n+1}$. Moreover, $M_j$ may be chosen so that $R^j$, the scalar curvature of $M_j$, satisfies $R^j\geq 2\kappa\left(1-\frac{1}{10j}\right) >\kappa$. Moreover, $M_\infty$ is not a length space and is not locally geodesically complete.
\end{theorem}
\printbibliography
\end{document}